\documentclass[12pt,a4paper,twoside]{amsart}
\usepackage{mathrsfs}
\usepackage{enumitem}
\usepackage{geometry}
\usepackage{needspace}
\usepackage{latexsym}
\usepackage{amsmath, amsfonts, amssymb, amsthm, amscd}

\usepackage{mathtools}

\usepackage{stackrel}
\usepackage[all]{xy}
\usepackage{fancyhdr}
\usepackage[hypertexnames=false,hidelinks]{hyperref}
\usepackage{color}
\usepackage[mathscr]{eucal}

\begin{document}
\newtheorem{theorem}{Theorem}[section]
\newtheorem{thmdef}{Theorem-definition}[section]
\newtheorem{lemma}[theorem]{Lemma}
\newtheorem{caution}[theorem]{Caution}
\newtheorem{definition}[theorem]{Definition}

\newtheorem{lemmadef}[theorem]{Lemma-definition}
\newtheorem{proposition}[theorem]{Proposition}
\newtheorem{corollary}[theorem]{Corollary}
\newtheorem{conjecture}[theorem]{Conjecture}
\newtheorem*{conjecture*}{Conjecture}
\newtheorem{observation}[theorem]{Observation}
\newtheorem{question}[theorem]{Question}
\newtheorem{problem}[theorem]{Problem}

\newtheorem*{definition*}{Definition}
\newtheorem*{caution*}{Caution}

\theoremstyle{definition}
\newtheorem{remark}[theorem]{Remark}
\theoremstyle{plain}

\newtheorem*{lemma*}{Lemma}

\renewcommand{\bar}{\overline}
\def\C{\mathbb{C}}
\def\R{\mathbb{R}}

\def\T{\mathrm{T}}
\def\X{\mathcal{X}}
\def\U{\mathcal{U}}
\def\P{\Phi}
\def\M{\mathcal{M}}
\def\Z{\mathcal{Z}_{m}}
\def\ZZ{\mathcal{Z}_{m}^{H}}
\def\d{\partial}
\def\cP{\mathscr P}
\def\bB{\mathbb B}
\def\TT{\mathcal{T}_{m}^{H}}
\def\GD{\Gamma \backslash D}
\def\mf{\mathfrak}
\def\ad{\mathrm{ad}\,}
\def\Ad{\mathrm{Ad}\,}
\def\mI{\mathscr{I}}
\def\mU{\mathscr{U}}
\newcommand{\Pp}{\mathbb P}
\newcommand{\OO}{\mathcal O}
\newcommand{\Pic}{\mathrm{Pic}}
\newcommand{\Fl}{\mathrm{Fl}}

\def\nca{\mathrm{nc1}}
\def\ncb{\mathrm{nc2}}

\def\i{\sqrt{-1}}

\title{Vanishing Theorems and Geometric Structures on Non-Classical Flag Domains}
\author{Kefeng Liu}
\address{Mathematical Sciences Research Center, Chongqing University of Technology, Chongqing 400054, China; \newline
Department of Mathematics, University of California at Los Angeles, Los Angeles, CA 90095-1555, USA}
\email{liu@math.ucla.edu}

\author{Yang Shen}
\address{Mathematical Sciences Research Center, Chongqing University of Technology, Chongqing 400054, China}
\email{syliuguang2007@163.com}

\begin{abstract}
Let $D=G_\R/V$ be a non-classical flag domain with $G_\R$ simple.
Although $D$ inherits a K\"ahler form from its projective compact
dual $\check D$, we prove that every closed positive $(1,1)$-current
on a smooth lattice quotient $X=\Gamma\backslash D$ vanishes,
even when $X$ is noncompact. We first prove the Green--Griffiths--Kerr 
conjecture on weights and deduce that every $G_\R$-invariant closed 
semipositive real $(1,1)$-form on $D$ vanishes.
The main analytic step is to prove that the average of the
lifted current over $\Gamma\backslash G_\R$ converges without
any growth assumption. As applications, we show that any lattice
quotient of $D$, smooth or singular, admits no K\"ahler current
and cannot be embedded as an open subset of a compact complex
space in Fujiki class $\mathcal C$.

When $G_\R$ is of Hermitian type, $G_\R/V$ carries a second
invariant complex structure $D'$. With this structure the same
smooth quotient $X'=\Gamma\backslash D'$ is projective for
cocompact torsion-free $\Gamma$ and quasi-projective for neat
arithmetic $\Gamma$ when $G_\R$ is linear.
We also construct topologically trivial nonhomogeneous
holomorphic line bundles on $SU(2,1)/H$ whose pullbacks to
every rational curve are trivial.
\end{abstract}

\maketitle

\tableofcontents

\parskip=5pt
\baselineskip=15pt

\vspace{-20pt}




\setcounter{section}{-1}

\section{Introduction}

Let $D=G_\R/V$ be a flag domain, an open orbit in a complex
projective flag variety $\check D=G_\C/B$, where $G_\R$ is connected,
noncompact and semisimple, and $V$ is the compact centralizer
of a subtorus of a compact Cartan subgroup $H$.
Let $K$ be a maximal compact subgroup containing $V$.
The projection
$$
p:D=G_\R/V\longrightarrow G_\R/K
$$
has compact flag fiber $K/V$. We call $D$ classical if an
invariant Hermitian symmetric complex structure on $G_\R/K$
makes $p$ holomorphic, and non-classical otherwise.

For a discrete subgroup $\Gamma\subset G_\R$, write
$X=\Gamma\backslash D$.
We call $\Gamma$ a lattice if $\Gamma\backslash G_\R$ has finite
Haar volume. A lattice need not be cocompact or torsion-free.

The main result of this paper is the following theorem.

\begin{theorem}[Theorem~\ref{lattice positivity obstruction}]
\label{intr main vanishing theorem}
Let $D=G_\R/V$ be non-classical with $G_\R$ simple, and let
$\Gamma\subset G_\R$ be a lattice.
Then:
\begin{enumerate}
\item[(i)] the quotient $X=\Gamma\backslash D$ admits no K\"ahler current;
\item[(ii)] when $X$ is smooth, every closed positive
$(1,1)$-current on $X$ is zero.
\end{enumerate}
\end{theorem}

On a possibly singular complex space, we consider closed positive
$(1,1)$-currents given locally by plurisubharmonic potentials.
Such a current $T$ is a \emph{K\"ahler current}
if every point has a neighborhood on which $T\geq\varepsilon\omega$
for some $\varepsilon>0$ and some smooth Hermitian form $\omega$,
induced from a local embedding in a complex manifold; see
\cite{DemaillyCADG} for the details.

Carlson and Toledo used harmonic maps to prove that compact
torsion-free quotients $X=\Gamma\backslash D$ of non-classical flag domains admit
no K\"ahler metric \cite[Theorem~0.1]{CT}. They also noted
the technical difficulty of extending this method to
noncompact quotients \cite[p.~52]{CT}.
For $G_\R$ simple and of adjoint type, Griffiths, Robles and
Toledo proved nonalgebraicity of quotients by infinite,
finitely generated discrete subgroups, using rational chains
and the Shafarevich map \cite{GRT}.
Their theorem does not exclude K\"ahler metrics on noncompact
quotients, as noted in \cite[p.~411]{CMP}.

The proof of Theorem~\ref{intr main vanishing theorem} contains the following algebraic and analytic ingredients.

The algebraic ingredient is the following assertion,
conjectured by Green, Griffiths and Kerr \cite[p.~78, footnote~9]{GGK}.

\begin{proposition}[Proposition~\ref{le D=0 conjecture}]
Let $D=G_\R/H$ be a non-classical flag domain with $G_\R$ simple,
and let $\lambda\in\sqrt{-1}\mf h_0^*$ be a nonzero real weight.
Then either there exists $\alpha\in\Delta_+^{\mathrm{c}}$
such that $(\lambda,\alpha)<0$, or there exists
$\beta\in\Delta_+^{\mathrm{nc}}$ such that $(\lambda,\beta)>0$.
\end{proposition}

Here $\Delta_+^{\mathrm c}$ and $\Delta_+^{\mathrm{nc}}$ denote
the positive compact and noncompact roots, respectively.
The proof in Appendix~\ref{lemma proof} is root-theoretic and
applies to real weights, without an integrality assumption.
On $G_\R/H$, real weights parametrize invariant
closed real $(1,1)$-forms; the proposition says that every
nonzero such form has a negative direction. Invariant currents
are smooth. Pullback along $G_\R/H\to G_\R/V$ therefore gives
the same vanishing for invariant closed positive $(1,1)$-currents
on $G_\R/V$.

The main analytic ingredient is to justify averaging on a noncompact
lattice quotient. For a $\Gamma$-invariant closed positive
$(1,1)$-current $\widetilde T$ on $D$, consider
$$
\mathcal A(\widetilde T)
=\int_{\Gamma\backslash G_\R}g^*\widetilde T\,d\mu(\Gamma g),
$$
where $\mu$ is the Haar probability measure.
Finite covolume alone does not ensure that this integral
converges. We first use compact flag cycles and
Proposition~\ref{le D=0 conjecture} to obtain finite averages
in the simple root directions. Using closedness on triples of
root vectors, we extend this bound to all positive roots by
induction on their height. Thus the average $\mathcal A(\widetilde T)$ exists
without a growth assumption on $\widetilde T$. 
It is an
invariant closed semipositive form, hence zero; positivity then
forces $\widetilde T=0$.

As applications of the above results, we obtain the following vanishing theorems
for holomorphic bundles.

\begin{theorem}[Theorem~\ref{D=0 conjecture} and
Corollary~\ref{sec7-meromorphic-consequences}]
\label{intr homogeneous vanishing}
Let $D=G_\R/V$ be non-classical with $G_\R$ simple, and let
$\Gamma\subset G_\R$ be a torsion-free discrete subgroup.
Put $X=\Gamma\backslash D$.
\begin{enumerate}
\item[(1)] If $\Gamma$ is cocompact, then
$$
H^0(X,\mathcal E_X)=0
$$
for every locally homogeneous vector bundle
$
\mathcal E_X=\Gamma\backslash(G_\R\times_V E)
$
induced by a nontrivial irreducible finite-dimensional
representation of $V$ on $E$.

\item[(2)] If $\Gamma$ is a lattice, not necessarily cocompact, then
$$
H^0(X,L_X)=0
$$
for every holomorphic line bundle $L_X$ on $X$ with
$L_X\not\simeq\mathcal O_X$.
\end{enumerate}
\end{theorem}

Part~(1), whose line-bundle case was also conjectured in
\cite[p.~78]{GGK}, follows directly from the weight inequality
in Proposition~\ref{le D=0 conjecture}, together with
Borel--Weil--Bott and the maximum principle.
The same argument gives the vanishing of all holomorphic
covariant tensors of positive degree
(Corollary~\ref{Omega on X}).

Part~(2) requires both the algebraic and the analytic arguments.
Theorem~\ref{main} (see also
Theorem~\ref{intr nonhomogeneous bundles} below) constructs
nonhomogeneous holomorphic line bundles on $SU(2,1)/H$,
disproving the homogeneity assertion in
\cite[pp.~240--241, ``Some speculations'', (i)(a)]{GGK}.
Thus arbitrary line bundles on a flag domain cannot be
described by weights.
To prove (2), we instead apply the analytic vanishing theorem
to the integration current of the zero divisor of a section.
Every nonzero section is therefore nowhere vanishing and
trivializes the line bundle.

These results also obstruct open embeddings into compact
complex spaces in Fujiki class $\mathcal C$.
Griffiths, Robles and Toledo record Koll\'ar's observation
that their nonalgebraicity theorem excludes Zariski open
embeddings into such spaces \cite{GRT}.
We include a proof using rational chains and the Shafarevich
map. For lattice quotients, current vanishing removes the condition
that the complement be analytic.

\begin{corollary}[Theorem~\ref{no Fujiki compactification} and
Theorem~\ref{lattice open embeddings}]
\label{intr open embeddings}
Let $D=G_\R/V$ be non-classical with $G_\R$ simple, and let
$\Gamma\subset G_\R$ be an infinite discrete subgroup.
Put $X=\Gamma\backslash D$.
\begin{enumerate}
\item[(i)] If $\Gamma$ is finitely generated and either
$\Gamma$ is torsion-free or $G_\R$ is of adjoint type, then $X$ is
not biholomorphic to a Zariski open subset of a compact
complex space in Fujiki class $\mathcal C$.
\item[(ii)] If $\Gamma$ is a lattice, then $X$ is not
biholomorphic to an open subset of any compact
complex space in Fujiki class $\mathcal C$.
\end{enumerate}
\end{corollary}

The lattice conclusion is interesting because $D$ is an open
subset, in the analytic topology, of the projective flag
variety $\check D$. Ample line bundles on $\check D$, together
with Hermitian metrics of positive curvature, restrict to $D$.
Its lattice quotients, however, cannot be embedded as open
subsets of compact complex spaces in Fujiki class $\mathcal C$.

The two assertions use different properties of the quotient.
For (i), a hypothetical compactification has a projective
modification. Rational chain connectedness and the Shafarevich map then
force a finite fundamental group, contradicting the infinitude
of $\Gamma$, after passage to a finite cover when needed.
For (ii), pulling back a K\"ahler form from a compact
K\"ahler modification produces a nonzero
$\Gamma$-invariant closed positive $(1,1)$-current on $D$.
The averaging theorem excludes this current even when the
boundary of $X$ is not analytic.

On a smooth lattice quotient $X$ as above, pseudo-effective
line bundles are unitary flat, and
every meromorphic map from $X$ to a K\"ahler manifold is constant
(Corollary~\ref{sec7-meromorphic-consequences}).
We also describe positive currents on modifications and prove
that current vanishing persists under small deformations of
compact torsion-free quotients.

Bakker--Brunebarbe--Tsimerman proved that the reduced closures
of period images are quasi-projective and that the Griffiths
$\mathbb Q$-line bundle is ample on them \cite{BBT23};
we gave another proof in \cite{LS26}.
These closures therefore admit projective compactifications.
For the non-classical lattice quotients considered here,
however, Theorem~\ref{lattice open embeddings} rules out
first compactifying the entire ambient quotient within
Fujiki class $\mathcal C$ and then taking the closure of
the period image.

There is also an obstruction to extending the Hodge metric.
On a smooth lattice quotient as above, the Hodge metric on
the smooth locus of a positive-dimensional period image cannot
be induced by a semipositive singular Hermitian metric on
any holomorphic line bundle over the entire quotient, by
Theorem~\ref{lattice positivity obstruction}(ii).

When $G_\R$ is of Hermitian type, the obstruction to algebraicity
depends on the invariant complex structure, as the next result
shows.

\begin{theorem}[Theorems~\ref{new complex D},~\ref{proj no Fujiki}
and Corollary~\ref{arithmetic complex structures}]
\label{intr new complex D}\label{intr proj no Fujiki}
Let $D=G_\R/V$ be non-classical, where $G_\R$ is semisimple
and of Hermitian type. There is a $G_\R$-invariant complex
structure $D'$ on the same smooth manifold, agreeing with $D$
on the fibers $K/V$, for which $p:D'\to G_\R/K$ is holomorphic.
For a discrete subgroup $\Gamma\subset G_\R$, put
$X=\Gamma\backslash D$ and $X'=\Gamma\backslash D'$.
\begin{enumerate}
\item[(i)] If $\Gamma$ is cocompact and torsion-free, then
$$
X\notin\mathcal C,\, X'\text{ is projective}.
$$
\item[(ii)] If $G_\R$ is simple and linear and $\Gamma$ is
a neat arithmetic subgroup, then
$$
X \text{ admits no K\"ahler current},\,
X'\text{ is quasi-projective}.
$$
\end{enumerate}
\end{theorem}

The construction changes the complex structure in the
noncompact tangent directions and preserves the compact flag
fibers. In (ii), the arithmetic Hermitian symmetric quotient
is quasi-projective by the Baily--Borel theorem \cite{BB66}.
The automorphic-bundle construction realizes $X'$ as an
algebraic flag bundle with projective projection over this base
\cite{Milne88}, so $X'$ is quasi-projective. Thus even in the
noncompact case, the same smooth manifold carries both a
quasi-projective complex structure and one with no K\"ahler
current.

Combining the above results, we obtain the following
sharp geometric contrast in the Hermitian case,
under the hypotheses of Theorem \ref{intr new complex D}.
Here the first and third rows represent analytic and
Zariski open embeddings, respectively, whereas the
middle row rules out even an analytic open embedding.
\begin{center}
\renewcommand{\arraystretch}{1.25}
\begin{tabular}{@{}ccl@{}}
$D$
& $\xhookrightarrow{\mathrm{open}}$
& $\check D$ (projective)
\\
$\Big\downarrow\mathrlap{\scriptstyle\,/\Gamma}$
& &
\\
$X=\Gamma\backslash D$
& $\overset{\mathrm{open}}{\not\hookrightarrow}$
& $Y$ (any compact $Y\in\mathcal C$)
\\
$\Big\updownarrow\mathrlap{\scriptstyle\,\mathrm{diff.}}$
& &
\\
$X'=\Gamma\backslash D'$
& $\xhookrightarrow{\mathrm{Zariski\ open}}$
& $\overline{X'}$ (projective).
\end{tabular}
\end{center}

The last section concerns holomorphic line bundles on the
flag domain itself. Homogeneous bundles do not exhaust them,
even among bundles that are topologically trivial.

\begin{theorem}[Theorem~\ref{main}]
\label{intr nonhomogeneous bundles}
On $D=SU(2,1)/H$, there is a family of pairwise nonisomorphic
topologically trivial holomorphic line bundles $L_t$,
$t\in\C^*$, whose pullbacks under every holomorphic map
$\mathbb P^1\to D$ are trivial.
For each $t$, the holomorphic isomorphism class of $L_t$ is
not preserved by the $SU(2,1)$ action, and $L_t$ admits no
smooth Hermitian metric with $SU(2,1)$-invariant Chern curvature.
\end{theorem}

The homogeneity statement proposed in
\cite[pp.~240--241, (i)(a)]{GGK} therefore fails without further
hypotheses. These examples also show that
triviality on rational curves does not detect the holomorphic
structure of a topologically trivial line bundle on $D$.

The paper is organized as follows.
Section~\ref{Pre} recalls the notation for flag domains
and the corresponding root decompositions.
Section~\ref{auto cohom} proves the vanishing theorems
for locally homogeneous bundles and holomorphic covariant
tensors on compact quotients.
Section~\ref{equ ch of D} gives equivalent characterizations
of non-classical flag domains in terms of weights, Lie
brackets and cycle chains.
In Section~\ref{hom line}, we prove that the averaging integral
converges on noncompact lattice quotients and establish
the vanishing of closed positive $(1,1)$-currents on smooth
lattice quotients.
Section~\ref{applications} contains applications to line
bundles, meromorphic maps and open embeddings into compact
spaces in Fujiki class $\mathcal C$.
In Section~\ref{ppl}, we study invariant complex structures
and show that, when $G_\R$ is of Hermitian type, changing
the complex structure gives projective and quasi-projective
quotients on the same underlying smooth manifolds.
In Section~\ref{app-nonhomogeneous}, we construct topologically
trivial nonhomogeneous line bundles on $SU(2,1)/H$, giving
counterexamples to the homogeneity assertion in \cite{GGK}.
The independent proof of Proposition~\ref{le D=0 conjecture}
is given in Appendix~\ref{lemma proof}.

\medskip
\noindent\textbf{Use of AI.}
The main theme and overall structure of this paper remain those
of the first version,
\href{https://arxiv.org/abs/2405.16536v1}{arXiv:2405.16536v1}.
Among the results added later, AI helped us search for the
nonhomogeneous line bundles in
Section~\ref{app-nonhomogeneous}.
We also used AI to shorten and clarify the preliminary
material in earlier versions, following the comments
received during previous submissions.
The authors are responsible for all statements and proofs
in this paper.

\section{Flag domains and root decompositions}\label{Pre}
We recall the notation for flag domains and roots, following \cite{GGK}.

Let $G_\R$ be a connected noncompact real semisimple Lie group,
$H$ a compact Cartan subgroup, and $V$ the compact centralizer of
a subtorus $S\subset H$. Choose a maximal compact subgroup $K$
containing $V$, so that
\begin{equation}\label{HVK}
H\subset V\subset K\subset G_\R.
\end{equation}
Since $V=Z_K(S)$ is the centralizer of a torus in the connected compact
group $K$, it is connected.
The given complex structure is induced by an open orbit
\begin{equation}\label{DinDcheck}
D=G_\R/V\subset\check D=G_\C/B,\qquad V=G_\R\cap B,
\end{equation}
where $G_\C$ is the complexification of $G_\R$, $B$ is parabolic,
and $\check D$ is the compact dual. The natural projection
\begin{equation}\label{DtoGK}
p:D=G_\R/V\longrightarrow G_\R/K
\end{equation}
has compact flag fiber $K/V$.
We call $D$ \emph{classical} if an invariant Hermitian symmetric
complex structure on $G_\R/K$ makes $p$ holomorphic, and
\emph{non-classical} otherwise. The group $G_\R$ is of
\emph{Hermitian type} when $G_\R/K$ is Hermitian symmetric.

We use a subscript $0$ for real Lie algebras and omit it for their
complexifications. A subscript $\C$ denotes a complexified group. Thus
$$
\mf h_0\subset\mf v_0\subset\mf k_0\subset\mf g_0,
\qquad
\mf h\subset\mf v\subset\mf k\subset\mf g,
$$
with $\mf v_0=\mf g_0\cap\mf b$.
The Cartan decomposition and its bracket relations are
$$
\mf g_0=\mf k_0\oplus\mf p_0,\qquad
[\mf k_0,\mf p_0]\subseteq\mf p_0,\qquad
[\mf p_0,\mf p_0]\subseteq\mf k_0.
$$
Put $\mf p=\mf p_0\otimes_\R\C$; then
$\mf g_c=\mf k_0\oplus\sqrt{-1}\mf p_0$ is a compact real form
of $\mf g$.

For $\Delta=\Delta(\mf g,\mf h)$, write
$$
\mf g=\mf h\oplus\bigoplus_{\alpha\in\Delta}\mf g_\alpha,
\qquad \mf g_\alpha=\C e_\alpha.
$$
The compact and noncompact roots are defined by
$$
\mf k=\mf h\oplus\bigoplus_{\alpha\in\Delta^{\mathrm c}}\mf g_\alpha,
\qquad
\mf p=\bigoplus_{\beta\in\Delta^{\mathrm{nc}}}\mf g_\beta,
\qquad \Delta=\Delta^{\mathrm c}\sqcup\Delta^{\mathrm{nc}}.
$$
Choose a Borel subalgebra $\mathfrak b_{\mathrm{Bor}}$ with
$\mf h\subset\mathfrak b_{\mathrm{Bor}}\subset\mf b$; it determines the
positive system $\Delta_+\subset\Delta$ by
\begin{equation}\label{Borel alg}
\mathfrak b_{\mathrm{Bor}}
=\mf h\oplus\bigoplus_{\alpha\in\Delta_+}\mf g_{-\alpha}.
\end{equation}
Set
\begin{align}
\Delta_+^{\mathrm c}&=\Delta^{\mathrm c}\cap\Delta_+,
\label{compact roots}\\
\Delta_+^{\mathrm{nc}}&=\Delta^{\mathrm{nc}}\cap\Delta_+.
\label{noncompact roots}
\end{align}
Let $\sqrt{-1}\mf h_0^*\subset\mf h^*$ denote the real span of $\Delta$,
with the positive definite inner product $(\cdot,\cdot)$ induced by the
Killing form $B_{\mf g}$. A \emph{real weight} is an element of
$\sqrt{-1}\mf h_0^*$; it need not be the differential of a character of
$H$. The weights defining homogeneous bundles are differentials of
characters of $H$. For the root inequalities and invariant closed
$(1,1)$-forms, we shall also allow nonintegral real weights.

For a Lie subalgebra $\mf h\subseteq\mf l\subseteq\mf g$, let
$\Delta(\mf l,\mf h)=\{\alpha:\mf g_\alpha\subseteq\mf l\}$,
with positive and negative parts $\Delta_\pm(\mf l,\mf h)$.
For a subspace $\mf s\subseteq\mf g$, put
$\Delta_{\mf s}=\{\alpha:e_\alpha\in\mf s\}$.

The nilradical of $\mf b$ and its opposite are
$$
\mf n_+=\bigoplus_{\alpha\in\Delta_+\setminus\Delta(\mf v,\mf h)}
\mf g_{-\alpha},\qquad
\mf n_-=\bigoplus_{\alpha\in\Delta_+\setminus\Delta(\mf v,\mf h)}
\mf g_\alpha,
$$
so $\mf g=\mf n_+\oplus\mf v\oplus\mf n_-=
\mf b\oplus\mf n_-$, and $\mf b=\mf v\oplus\mf n_+$ with Levi
factor $\mf v$.
Writing $\mf k_\pm=\mf n_\pm\cap\mf k$ and
$\mf p_\pm=\mf n_\pm\cap\mf p$, we have
$$
\mf k=\mf k_+\oplus\mf v\oplus\mf k_-,\qquad
\mf p=\mf p_+\oplus\mf p_-.
$$
The complexified tangent bundle is
$$
\mathrm T^{\C}D=G_\R\times_V(\mf g/\mf v)
=G_\R\times_V(\mf n_+\oplus\mf n_-),
$$
where complex conjugation with respect to $\mf g_0$ interchanges
$\mf n_+$ and $\mf n_-$. Since $D$ is open in $\check D=G_\C/B$, the holomorphic
tangent space at the base point is $\mf g/\mf b\simeq\mf n_-$ as a
$V$-module, and
$$
\mathrm T^{1,0}D=G_\R\times_V\mf n_-
=G_\R\times_V(\mf k_-\oplus\mf p_-)\subset\mathrm T^{\C}D
$$
as smooth complex homogeneous bundles. Here $\mf k_-$ is tangent to
$K/V$; the splitting need not be holomorphic.

The differential of $p$ at the base point maps
$\mf n_-=\mf k_-\oplus\mf p_-$ onto $\mf p_-\subset\mf p\simeq\mf g/\mf k$.
If $D$ is classical, $p$ is holomorphic, so $\mf p_-$ is
the holomorphic tangent space of the Hermitian symmetric space
$G_\R/K$ at the base point; by the Harish-Chandra decomposition
\cite[Chapter~VII, Section~9]{knapp}, $\mf p_-$ and $\mf p_+=\overline{\mf p_-}$
are then $\mf k$-invariant abelian subalgebras of $\mf g$. This is used
in Theorem~\ref{equ of non-classical D} and Section~\ref{ppl}.

For every $D$, the definitions above give
$$
\Delta_+^{\mathrm c}=\Delta_{\mf k_-}\sqcup
\Delta_+(\mf v,\mf h),\qquad
\Delta_+^{\mathrm{nc}}=\Delta_{\mf p_-}.
$$

\section{Vanishing theorems for locally homogeneous bundles}
\label{auto cohom}

A locally homogeneous bundle on $X=\Gamma\backslash D$ is obtained
by descending a $G_\R$-homogeneous bundle on $D$. For a
finite-dimensional $V$-representation, we use the holomorphic
structure determined by extending its differential to $\mf b$,
with the nilradical acting trivially. When the representation
extends to the Levi subgroup of $B$, this is the restriction of
the usual homogeneous holomorphic bundle on $G_\C/B$.

The following lemma reduces the study of sections of a locally
homogeneous vector bundle to that of a line bundle on
$\Gamma\backslash G_\R/H$. The quotient need not be compact.

\begin{lemma*}[Reduction Lemma]
Let $D=G_\R/V\subset G_\C/B$ be a flag domain as above, where $V$ is
the compact centralizer of a subtorus of the compact Cartan subgroup $H$.
Choose the Borel subgroup $\tilde B\subset B$ with Lie algebra
$\mathfrak b_{\mathrm{Bor}}$ in \eqref{Borel alg}, and give $G_\R/H$
the complex structure induced by its open embedding in $G_\C/\tilde B$.
For a torsion-free discrete subgroup $\Gamma\subset G_\R$, which may
be trivial, the natural projection
$$\pi:\,Y=\Gamma\backslash G_\R/H\longrightarrow X=\Gamma\backslash G_\R/V$$
is a proper holomorphic map. Then the following hold:

(i) For every holomorphic line bundle $L$ on $X$,
$$H^{0}(Y,\pi^{*}L)=H^{0}(X,L).$$

(ii) If $$\mathcal E_X = \Gamma \backslash \left( G_\R \times_V E\right)$$ is induced by a finite-dimensional irreducible complex representation of $V$ with highest weight $\lambda$ relative to $\Delta_+(\mf v,\mf h)$, then there exists a locally homogeneous line bundle $$\mathcal L_Y=\Gamma \backslash \left( G_\R \times_H \C\right)$$ on $Y$ with weight $\lambda$ such that $R^0\pi_{*}\mathcal L_Y\simeq\mathcal E_X$ and
$$H^0(Y,\mathcal L_Y)\simeq H^0(X, R^0\pi_*\mathcal L_Y)\simeq H^0(X, \mathcal E_X).$$
\end{lemma*}
\begin{proof}
The projection is a holomorphic fiber bundle with compact connected
fiber
$$
F=V/H\simeq V_\C/(V_\C\cap\tilde B).
$$
Since $H^0(F,\mathcal O_F)=\C$, holomorphic functions on the total
space are constant on the fibers, and local trivializations give
$\pi_*\mathcal O_Y=\mathcal O_X$. The projection formula then gives
$\pi_*\pi^*L=L$, proving (i).

For (ii), let $\chi_\lambda:H\to\C^*$ be the character with
differential $\lambda$, and set
$$
\mathcal L_Y=\Gamma\backslash
\bigl(G_\R\times_H\C_{\lambda}\bigr).
$$
Its holomorphic structure is defined by extending $\lambda$ to
$\mathfrak b_{\mathrm{Bor}}$, trivially on the nilradical.
By \eqref{Borel alg}, $\tilde B$ corresponds to the negative roots.
The Borel--Weil--Bott theorem therefore gives
$$
H^0(F,\mathcal L_Y|_F)\simeq E,
\qquad H^p(F,\mathcal L_Y|_F)=0\quad(p>0)
$$
as $V$-modules; see \cite[pp.~91--93]{GGK}.
These isomorphisms are $V$-equivariant, so the transition maps of
$\pi_*\mathcal L_Y$ agree with those of $\mathcal E_X$ in local
holomorphic trivializations of $\pi$. Hence
$\pi_*\mathcal L_Y\simeq\mathcal E_X$ as holomorphic bundles, and
$$
H^0(Y,\mathcal L_Y)=H^0(X,\pi_*\mathcal L_Y)
\simeq H^0(X,\mathcal E_X).
$$
\end{proof}
Green, Griffiths and Kerr conjectured the line-bundle case of the
following theorem in \cite[p.~78]{GGK}, where they proved it for
regular weights. The Reduction Lemma gives the corresponding
statement for vector bundles.

\begin{theorem}\label{D=0 conjecture}
Let $D=G_\R/V$ be a non-classical flag domain with $G_\R$ simple.
Let $\Gamma\subset G_\R$ be a cocompact torsion-free discrete
subgroup, and put $X=\Gamma\backslash D$. For every locally
homogeneous vector bundle
$$\mathcal E_X = \Gamma \backslash \left( G_\R \times_V E\right)$$
induced by a finite-dimensional irreducible representation
$\rho:V\to\mathrm{GL}(E)$ with highest weight $\lambda\ne0$,
we have
\begin{equation}\label{D=0 conjecture 1}
H^0(X,\mathcal E_X)=0.
\end{equation}
\end{theorem}
\begin{proof}
The domain $G_\R/H$ is classical if and only if $D$ is classical,
since the two projections to $G_\R/K$ have the same noncompact
holomorphic tangent directions. By the Reduction Lemma, it suffices
to prove that
\begin{equation}\label{D=0 conjecture 2}
 H^0(Y,\mathcal L_Y)=0
\end{equation}
for a locally homogeneous line bundle $\mathcal L_Y$ on $Y=\Gamma \backslash G_\R/H$ with weight $\lambda \neq 0$.

Equip $\mathcal L_Y$ with its invariant Hermitian metric. Its
Chern curvature is given by \cite[pp.~74--77]{GGK}
\begin{equation}\label{curv of E}
\Theta= \sum_{\alpha\in \Delta_+^{\mathrm{c}}}(\lambda,\alpha) \omega^\alpha \wedge \bar{\omega}^\alpha-\sum_{\beta\in \Delta_+^{\mathrm{nc}}}(\lambda,\beta) \omega^\beta \wedge \bar{\omega}^\beta.
\end{equation}
The forms $\omega^\alpha$ are locally induced by the covectors
dual to $e_\alpha$; the sum is a global form on $Y$.

If $s$ is a nonzero holomorphic section, compactness of $Y$
provides a point $x_0$ where $\|s\|^2$ attains a positive maximum.
The section does not vanish near $x_0$, and there
$$
\Theta=\bar\partial\partial\log\|s\|^2.
$$
Consequently $\sqrt{-1}\Theta$ is semipositive at $x_0$.
Proposition~\ref{le D=0 conjecture} gives either a compact root
with $(\lambda,\alpha)<0$ or a noncompact root with
$(\lambda,\beta)>0$. In either case \eqref{curv of E} shows that
$\sqrt{-1}\Theta$ has a negative direction at every point,
a contradiction.
\end{proof}

The following assertion about roots was conjectured in
\cite[p.~78]{GGK}. In Appendix~\ref{lemma proof}, we assume that
the inequalities fail and construct a nonzero ideal to obtain a
contradiction. Remark~\ref{full flag direct equivalence} shows
how this assertion implies the bracket-generation result of
\cite[Lemma~2.6]{GRT}.

\begin{proposition}\label{le D=0 conjecture}
Let $D=G_\R/H$ be a non-classical flag domain with $G_\R$ simple,
and let $\lambda\in\sqrt{-1}\mf h_0^*$ be a nonzero real weight.
Then either there exists $\alpha\in\Delta_+^{\mathrm{c}}$
such that $(\lambda,\alpha)<0$, or there exists
$\beta\in\Delta_+^{\mathrm{nc}}$ such that $(\lambda,\beta)>0$.
\end{proposition}

Under the hypotheses of Theorem~\ref{D=0 conjecture},
$\Gamma$-invariant sections on $D$ descend to $X$, giving
$H^0_{\Gamma}(D,\mathcal E)\simeq H^0(X,\mathcal E_X)$. Thus
the theorem also gives the vanishing of automorphic sections;
see \cite{WW}. It further implies the following vanishing of
holomorphic covariant tensors.

\begin{corollary}\label{Omega on X}
Let $D=G_\R/V$ be a non-classical flag domain with $G_\R$ simple,
and let $\Gamma\subset G_\R$ be a cocompact torsion-free discrete
subgroup. Put $X=\Gamma\backslash D$. Then
$$
H^0\bigl(X,(\mathrm T^*X)^{\otimes m}\bigr)=0
\qquad\text{for every }m\ge1.
$$
In particular, $H^0(X,\Omega_X^p)=0$ for $1\le p\le\dim X$.
\end{corollary}
\begin{proof}
The $B$-module $((\mf g/\mf b)^*)^{\otimes m}$ has a finite
filtration whose irreducible quotients are trivial on the unipotent
radical of $B$. It induces a holomorphic filtration of
$(\mathrm T^*X)^{\otimes m}$ by locally homogeneous subbundles,
with quotients associated with irreducible representations of $V$.
Their highest weights occur among
$$
-(\alpha_1+\cdots+\alpha_m),\qquad
\alpha_i\in\Delta_+\setminus\Delta(\mf v,\mf h),
$$
and are nonzero, since each $\alpha_i$ is positive.
Applying Theorem~\ref{D=0 conjecture} to the successive quotients
and using the associated exact sequences gives the assertion.
The exterior powers are holomorphic direct summands of the
corresponding tensor powers.
\end{proof}

The same argument applies to differential forms with values in
a locally homogeneous bundle.
\begin{corollary}
Assume the hypotheses of Theorem~\ref{D=0 conjecture}, and let
$\mathcal E_X$ be induced by an irreducible $V$-representation $E$.
If no weight $\mu$ of $E$ equals $\alpha_1+\cdots+\alpha_p$ for
mutually distinct roots
$$
\alpha_1,\ldots,\alpha_p\in
\Delta_+\setminus\Delta(\mf v,\mf h),
$$
then $H_{\bar\partial}^{p,0}(X,\mathcal E_X)=0$.
\end{corollary}
\begin{proof}
By the Dolbeault theorem,
$$
H_{\bar\partial}^{p,0}(X,\mathcal E_X)
=H^0(X,\Omega_X^p\otimes\mathcal E_X).
$$
The weights of $\bigwedge^p\mf n_-^*\otimes E$ are
$\mu-(\alpha_1+\cdots+\alpha_p)$, and none is zero by assumption.
The holomorphic filtration used above therefore has no trivial
irreducible quotient. The result follows from
Theorem~\ref{D=0 conjecture}.
\end{proof}
We next relate these vanishing results to the geometry of $D$;
the proof of Proposition~\ref{le D=0 conjecture} is deferred to
Appendix~\ref{lemma proof}.

\section{Characterizations of non-classical flag domains}\label{equ ch of D}
Throughout this section $G_\R$ is simple. 

The base cycle
$Z_o=K/V\simeq K_\C/(K_\C\cap B)$ is a compact flag variety in $D$.
Following \cite{GRT}, put
$$
\mathscr U=\{gK_\C\in G_\C/K_\C:gZ_o\subset D\}_0,
\qquad Z_u=gZ_o\quad(u=gK_\C),
$$
where the subscript $0$ denotes the connected component
containing $eK_\C$; different parameters may
represent the same cycle. We say that $D$ is \emph{cycle chain
connected} if any two points are joined by a finite chain of these
cycles, with consecutive members intersecting.

\begin{theorem}\label{equ of non-classical D}\label{conn of D}
For a flag domain $D=G_\R/V$ with $G_\R$ simple, the following
conditions are equivalent:
\begin{enumerate}
\item[(i)] $D$ is non-classical.
\item[(ii)] For every nonzero real weight $\lambda$ dominant on
$\Delta_+(\mf v,\mf h)$, either there exists
$\alpha\in\Delta_+^{\mathrm{c}}\setminus\Delta(\mf v,\mf h)$
such that $(\lambda,\alpha)<0$, or there exists
$\beta\in\Delta_+^{\mathrm{nc}}$ such that $(\lambda,\beta)>0$.
\item[(iii)] $\mf p_-$ is contained in the Lie algebra generated by
$\mf k_-\oplus\mf p_+$.
\item[(iv)] $D$ is cycle chain connected.
\end{enumerate}
These conditions are also equivalent to $H^0(D,\mathcal O_D)=\C$.
\end{theorem}
\begin{proof}
The proof of (i)$\Rightarrow$(ii) is given in
Appendix~\ref{lemma proof} for general $V$. It does not use cycle
chains. Here dominance means $(\lambda,\gamma)\geq0$ for all
$\gamma\in\Delta_+(\mf v,\mf h)$; no integrality is assumed.
The implication (i)$\Rightarrow$(iii) is the bracket-generation
assertion of \cite[Lemma~2.6]{GRT}; an independent root-theoretic
proof is given in Remark~\ref{full flag direct equivalence}.
The implication (i)$\Rightarrow$(iv) is due to
\cite[Proposition~3.1]{Huc}; see also \cite[Theorem~1.2]{GRT}.

Conversely, suppose $D$ is classical.
The weight of the canonical bundle of $G_\R/K$ is
$\lambda=-\sum_{\beta\in\Delta_+^{\mathrm{nc}}}\beta$. It is nonzero
and vanishes on compact roots. Since the invariant metric on this
canonical bundle has positive curvature, $\lambda$ pairs negatively
with every positive noncompact root. Thus (ii) fails.
Since $\mf p_+$ is abelian and $\mf k$-invariant,
$\mf k_-\oplus\mf p_+$ is a Lie subalgebra whose intersection
with $\mf p_-$ is zero. Thus (iii) fails. Every compact cycle has
constant image under the holomorphic projection to the bounded
symmetric domain $G_\R/K$,
so (iv) also fails.
Finally, holomorphic functions are constant on each compact cycle,
and hence on $D$ if (iv) holds. If $D$ is classical, nonconstant
holomorphic functions on $G_\R/K$ pull back to $D$.
\end{proof}

\begin{corollary}\label{subbundle restriction and forms}
Under the equivalent conditions of Theorem~\ref{equ of non-classical D},
for every holomorphic subbundle
$E\hookrightarrow\mathcal O_{\check D}^{\,N}$, restriction induces
an isomorphism
$$
H^0(\check D,E)\xrightarrow{\;\sim\;}H^0(D,E|_D).
$$
In particular, for every integer $r>0$,
$$
H^0(D,\Omega_D^r)=H^0(X,\Omega_X^r)=0,
$$
where $X=\Gamma\backslash D$ for any torsion-free discrete subgroup
$\Gamma\subset G_\R$, not necessarily cocompact.
\end{corollary}
\begin{proof}
Write
$$
0\longrightarrow E\xrightarrow{j}\mathcal O_{\check D}^{\,N}
\xrightarrow{q}\mathcal Q\longrightarrow0.
$$
For $s\in H^0(D,E|_D)$, the components of $j(s)$ are constant
since $H^0(D,\mathcal O_D)=\C$. Their value $c\in\C^N$ defines
a constant section on $\check D$. Since $q(c)$ vanishes on $D$,
the identity theorem gives $q(c)=0$ on $\check D$. Thus $s$
extends uniquely to a global section of $E$.

Since $G_\C$ acts transitively on $\check D$, the action map
$\mf g\otimes\mathcal O_{\check D}\to\T\check D$ is surjective.
Dualizing and taking exterior powers, we obtain
$$
\Omega_{\check D}^r\hookrightarrow
\bigwedge\nolimits^r\mf g^*\otimes\mathcal O_{\check D}.
$$
Compact flag varieties have no positive-degree holomorphic forms,
so restriction gives $H^0(D,\Omega_D^r)=0$. Injectivity of
pullback by $\pi:D\to X$ gives the assertion for $X$.
\end{proof}

For compact quotients, we also have the following criteria.

\begin{corollary}\label{compact quotient criteria}
Let $D=G_\R/V$ have $G_\R$ simple, and let
$X=\Gamma\backslash D$ for a cocompact torsion-free discrete
subgroup $\Gamma\subset G_\R$. Then $D$ is non-classical if and
only if any one of the following conditions holds:
\begin{enumerate}
\item[(a)] $H^0(X,\Gamma\backslash\mathcal E_\lambda)=0$ for every
irreducible homogeneous bundle $\mathcal E_\lambda$ with
$\lambda\ne0$, equipped with the holomorphic structure of
Section~\ref{auto cohom}.
\item[(b)] $H^0(X,(\mathrm T^*X)^{\otimes m})=0$ for every $m\geq1$.
\item[(c)] $H^0(X,L_X)=0$ for every nontrivial holomorphic line bundle
$L_X$ on $X$.
\item[(d)] $H^0(X,\mathcal M_X)=\C$, where $\mathcal M_X$ is the
sheaf of meromorphic functions.
\end{enumerate}
\end{corollary}
\begin{proof}
For non-classical $D$, (a) and (b) follow from
Theorem~\ref{D=0 conjecture} and Corollary~\ref{Omega on X};
(c) and (d) follow from Corollary~\ref{sec7-meromorphic-consequences}.
If $D$ is classical, $\tilde p:X\to B=\Gamma\backslash G_\R/K$
is holomorphic, and $B$ is positive-dimensional and projective with
ample $K_B$. For $q\gg0$, a nonzero section of $K_B^{\otimes q}$
pulls back to a nonzero section of the locally homogeneous line bundle
$\tilde p^*K_B^{\otimes q}$, whose weight $q\lambda$ is nonzero by the
calculation of $\lambda$ above. This line bundle
is nontrivial, since $\tilde p_*\mathcal O_X=\mathcal O_B$ and the
projection formula would otherwise trivialize $K_B^{\otimes q}$.
Writing $m_0=\dim_\C B$, the cotangent sequence also gives
$$
\tilde p^*K_B^{\otimes q}\hookrightarrow
(\mathrm T^*X)^{\otimes qm_0}.
$$
Thus (a)--(c) fail. Nonconstant meromorphic functions on $B$ pull
back to $X$, so (d) fails as well.
\end{proof}

Following the terminology of \cite[p.~237]{GGK}, here
\emph{rationally connected} means connected by chains of
images of holomorphic maps $\mathbb P^1\to X$, without a compactness
condition on their deformation spaces.

\begin{corollary}\label{Mum conj X}
Let $D=G_\R/V$ have $G_\R$ simple, and let
$\Gamma\subset G_\R$ be a cocompact torsion-free discrete subgroup.
Put $X=\Gamma\backslash D$.
Then $X$ is rationally connected if and only if
\begin{equation}\label{Mum=0}
H^0\bigl(X,(\mathrm T^*X)^{\otimes m}\bigr)=0
\qquad\text{for every }m\geq1.
\end{equation}
\end{corollary}
\begin{proof}
By Corollary~\ref{compact quotient criteria}, \eqref{Mum=0} holds
if and only if $D$ is non-classical. In this case the compact flag
cycles are rationally connected, so their chains descend to
rational chains in $X$. If $D$ is classical, every map
$\mathbb P^1\to X$ lifts to $D$ and has constant projection to
$G_\R/K$. Thus every rational chain lies in a fiber of $\tilde p$,
and $X$ is not rationally connected.
\end{proof}
For comparison, smooth projective rationally connected complex
varieties have no nonzero holomorphic covariant tensors of positive
degree \cite[Chapter~IV, Corollary~3.8]{Kol13}. The non-classical
quotients considered here lie outside Fujiki class $\mathcal C$
by Theorem~\ref{comact X not ddbar}.

\begin{remark}\label{full flag direct equivalence}
The equivalence of (ii) and (iii) in
Theorem~\ref{equ of non-classical D} can also be proved directly
using Lie algebras.

Let $\mf l$ be the Lie algebra generated by
$\mf k_-\oplus\mf p_+$.
Both summands are $\mf v$-invariant, so
$[\mf v,\mf l]\subseteq\mf l$.
Thus $\mf l+\mf v$ is a Lie subalgebra, and $\mf l$ is
$\mf h$-stable.

Assume (ii), and let $\beta$ be a noncompact simple root.
Since the Levi root system is generated by compact simple
roots, $-\beta$ is dominant on $\Delta_+(\mf v,\mf h)$.
If $(\beta,\gamma)<0$ for some
$\gamma\in\Delta_+^{\mathrm{nc}}$, then
$\alpha=\beta+\gamma$ is a positive compact root outside
$\Delta(\mf v,\mf h)$: its simple-root expansion contains
$\beta$.
Otherwise, applying (ii) to $-\beta$ gives
$\alpha\in\Delta_+^{\mathrm c}\setminus\Delta(\mf v,\mf h)$
with $(\beta,\alpha)>0$.
Then $\alpha-\beta$ is a noncompact root. It is positive, since
otherwise $\beta=\alpha+(\beta-\alpha)$ would express the simple
root $\beta$ as a sum of two positive roots.
In either case,
\begin{equation}\label{simple nc=c-nc}
\beta=\alpha-\beta',
\qquad
e_\beta=c_\beta[e_\alpha,e_{-\beta'}]\in\mf l,
\qquad c_\beta\ne0,
\end{equation}
for a positive compact root $\alpha$ outside the Levi root
system and a positive noncompact root $\beta'$.

Every compact simple root vector belongs to $\mf k_-$
or $\mf v$. Hence $\mf l+\mf v$ contains all positive
simple root vectors, and therefore all positive root vectors
\cite[Section~10.2]{Hum}.
Thus $\mf p_-\subseteq\mf l+\mf v$.
Since $\mf l$ is $\mf h$-stable, it is a sum of weight spaces,
whereas $\mf v$ contains only compact root spaces and $\mf h$.
Thus the noncompact root spaces in $\mf l+\mf v$ belong to
$\mf l$, so $\mf p_-\subseteq\mf l$. This proves (iii).

Conversely, assume (iii) and suppose that (ii) fails.
There is then a nonzero real weight $\lambda$ satisfying
$$
(\lambda,\alpha)\geq0
\quad(\alpha\in\Delta_+^{\mathrm c}),
\qquad
(\lambda,\beta)\leq0
\quad(\beta\in\Delta_+^{\mathrm{nc}}),
$$
where the first inequality on the Levi roots follows from
dominance.
All roots occurring in the generators of $\mf l$ pair
nonnegatively with $\lambda$.
Since weights add under brackets, the same holds for every
root occurring in $\mf l$.
Condition (iii) therefore gives
$$
0\geq(\lambda,\beta)\geq0
\qquad(\beta\in\Delta_+^{\mathrm{nc}}).
$$
Thus $\lambda$ vanishes on all noncompact roots.
The identity $\mf k=[\mf p,\mf p]$ in \eqref{k=p+p}
then forces it to vanish on all compact roots as well.
Since the roots span the real weight space, $\lambda=0$,
a contradiction.
\end{remark}

\section{Averaging positive currents on lattice quotients}
\label{hom line}
We retain the notation of Section~\ref{Pre} and write
$\mathcal P_d^{1,1}(M)$ for the cone of closed positive
$(1,1)$-currents on a complex analytic space $M$.
We begin with invariant forms.

\begin{lemma}\label{sec7-invariant-semipositive}
Let $D=G_\R/V$ be a non-classical flag domain with $G_\R$ simple.
Then every $G_\R$-invariant closed semipositive real $(1,1)$-form on $D$
is zero.
\end{lemma}
\begin{proof}
Let $\eta$ be such a form, and use the same notation for the value
of its pullback to $G_\R$ at the identity. Since the Killing form
$B_{\mf g}$ is nondegenerate, there is an endomorphism $A$ of
$\mf g_0$ such that
$$
\eta(X,Y)=B_{\mf g}(AX,Y),\qquad X,Y\in\mf g_0.
$$
Since $\eta$ is alternating, $A$ is skew-adjoint for $B_{\mf g}$.
The cocycle identity $d\eta=0$, together with invariance of
$B_{\mf g}$, therefore gives
$$
B_{\mf g}\bigl(A[X,Y]-[AX,Y]-[X,AY],W\bigr)=0
\qquad(X,Y,W\in\mf g_0).
$$
Hence $A$ is a derivation. Every derivation of a semisimple Lie algebra
is inner, so $A=\operatorname{ad}Z$ for some $Z\in\mf g_0$, and
$$
\eta(X,Y)=B_{\mf g}(Z,[X,Y]).
$$
Since $\eta$ is pulled back from $D=G_\R/V$, we have
$$
0=\eta(U,Y)=B_{\mf g}(Z,[U,Y])=B_{\mf g}([Z,U],Y)
\qquad(U\in\mf v_0,\ Y\in\mf g_0).
$$
Nondegeneracy of $B_{\mf g}$ gives $[Z,\mf v_0]=0$.
Since $\mf h_0\subset\mf v_0$ and the centralizer of
$\mf h_0$ is itself, $Z\in\mathfrak z(\mf v_0)\subset\mf h_0$.

With the root vectors of \eqref{curv of E}, normalized as in
\cite[(3.6)]{GS}, the form $\eta$ has the local expression
\begin{equation}\label{sec7-root-form}
\eta=
\sqrt{-1}\sum_{\alpha\in\Delta_+^{\mathrm c}\setminus\Delta(\mf v,\mf h)}
(\lambda,\alpha)\,\omega^\alpha\wedge\bar\omega^\alpha
-\sqrt{-1}\sum_{\beta\in\Delta_+^{\mathrm{nc}}}
(\lambda,\beta)\,\omega^\beta\wedge\bar\omega^\beta,
\end{equation}
where $\lambda=\sqrt{-1}B_{\mf g}(Z,\cdot)|_{\mf h_0}
\in\sqrt{-1}\mf h_0^*$ is the real weight corresponding to $Z$.

Since $[Z,\mf v_0]=0$, we have
$(\lambda,\gamma)=0$ for $\gamma\in\Delta(\mf v,\mf h)$.
If $\lambda\ne0$, Proposition~\ref{le D=0 conjecture}, applied
on $G_\R/H$, gives a negative direction in
\eqref{sec7-root-form}, contradicting semipositivity.
Hence $\lambda=0$ and $\eta=0$.
\end{proof}

\begin{remark}\label{sec7-classicality}
For $G_\R$ simple, $D$ is classical if and only if it admits a nonzero
invariant closed semipositive real $(1,1)$-form. This follows from
Lemma~\ref{sec7-invariant-semipositive} and, in the classical case,
the pullback of an invariant K\"ahler form on $G_\R/K$.
On a non-classical $D$, every nonzero invariant closed real
$(1,1)$-form is therefore indefinite at every point.

A holomorphic line bundle on a non-classical $D$ admits a smooth
Hermitian metric with invariant semipositive Chern curvature if and
only if it is holomorphically trivial, since such curvature vanishes
and $D$, which retracts onto the flag manifold $K/V$, is simply
connected. Invariance is essential: positive
curvature metrics on ample line bundles over $\check D$ restrict to $D$.
\end{remark}

For compact quotients, averaging presents no convergence problem.
In the next lemma $G_\R$ may be semisimple.
Let $\Gamma\subset G_\R$ be a cocompact torsion-free discrete subgroup,
put $X=\Gamma\backslash D$, and let $\pi:D\to X$ be the projection.
The compact manifold $\Gamma\backslash G_\R$ carries a normalized
right $G_\R$-invariant measure $\mu$. For a $\Gamma$-invariant current
$T$ on $D$, define its weak average by
\begin{equation}\label{sec7-average}
\mathcal A(T)=\int_{\Gamma\backslash G_\R}g^*T\,d\mu(\Gamma g).
\end{equation}
More explicitly, if $T$ has bidegree $(p,q)$ and $n=\dim_{\C}D$, then
for every $\varphi\in C_c^\infty(D,\Lambda^{n-p,n-q}T^*D)$,
$$
\langle\mathcal A(T),\varphi\rangle
=\int_{\Gamma\backslash G_\R}
\langle T,(g^{-1})^*\varphi\rangle\,d\mu(\Gamma g).
$$
For smooth forms, this agrees with the ordinary average.

\begin{lemma}\label{sec7-averaging-properties}
Let $\Gamma\subset G_\R$ be a cocompact torsion-free discrete
subgroup, and let $T$ be a $\Gamma$-invariant current on $D$.
The averaging operator \eqref{sec7-average} has the following
properties.
\begin{enumerate}
\item[(i)] $\mathcal A(T)$ is a smooth $G_\R$-invariant form.
Averaging preserves reality and bidegree and commutes with
$d$, $\partial$, and $\bar\partial$.
\item[(ii)] Averaging preserves positivity, and the average of a nonzero
positive $(1,1)$-current is nonzero.

\item[(iii)] For any lattice $\Gamma\subset G_\R$, not necessarily
cocompact or torsion-free, and any $\Gamma$-invariant current $T$,
the weighted average
$$
\mathcal A_\chi(T)=\int_{\Gamma\backslash G_\R}\chi(g)g^*T\,d\mu(g),
\qquad \chi\in C_c^\infty(\Gamma\backslash G_\R),
$$
is smooth, preserves bidegree, and commutes with
$d$, $\partial$, and $\bar\partial$. It is real if $T$ and $\chi$
are real, and positive if $T$ is positive and $\chi\geq0$.
\end{enumerate}
\end{lemma}
\begin{proof}
We first prove (iii). Left translation by $\Gamma$ on $G_\R$ is
free, so $Y=\Gamma\backslash G_\R$ is smooth even when $\Gamma$
has torsion. Since $T$ is $\Gamma$-invariant, $g^*T$ depends only
on $\Gamma g$.
Using local sections of $G_\R\to Y$ and a partition of unity on
$\operatorname{supp}\chi$, we express $\mathcal A_\chi(T)$ as a
finite sum of operators of the form
$$
S_fT=\int_{G_\R}f(g)g^*T\,dg,\qquad f\in C_c^\infty(G_\R).
$$
Since $G_\R$ acts transitively on $D$, the map
$$
G_\R\times D\longrightarrow D\times D,\qquad
(g,x)\longmapsto(gx,x),
$$
is a submersion. Integration along its fibers gives a smooth
kernel $K_f(y,x)$ such that
$$
(S_fT)(x)=\langle T(y),K_f(y,x)\rangle_y.
$$
For a $(p,q)$-current $T$, the kernel has bidegree
$(n-p,n-q)$ in $y$ and $(p,q)$ in $x$, where $n=\dim_\C D$.
For $x$ in a compact set $C\subset D$, its support in $y$
is contained in the fixed compact set $(\operatorname{supp}f)C$.
We may therefore differentiate inside the pairing:
$$
\partial_x^\alpha(S_fT)(x)
=\langle T(y),\partial_x^\alpha K_f(y,x)\rangle_y,
$$
where derivatives are taken coefficientwise in local coordinates.
Thus $S_f$ sends currents continuously to smooth forms on
compact subsets of $D$. The remaining assertions of (iii)
follow because holomorphic pullback preserves bidegree and
commutes with $d$, $\partial$ and $\bar\partial$; integration
against a real, respectively nonnegative, weight preserves
reality, respectively positivity.

When $Y$ is compact, we may take $\chi=1$ to obtain a smooth form
$\mathcal A(T)$. For $a\in G_\R$, right invariance of $\mu$ gives
$$
a^*\mathcal A(T)
=\int_{\Gamma\backslash G_\R}(ga)^*T\,d\mu(\Gamma g)
=\mathcal A(T).
$$
This together with (iii) implies (i).

Holomorphic pullback and integration preserve positivity.
If $T\ne0$ is positive, choose a compactly supported smooth
strongly positive test form $\varphi$ with
$\langle T,\varphi\rangle>0$.
The function $\Gamma g\mapsto\langle g^*T,\varphi\rangle$
is continuous, nonnegative and positive at $\Gamma e$.
Its integral is therefore positive, proving (ii).
\end{proof}

For a nonuniform lattice, finite Haar measure alone does not
ensure convergence of \eqref{sec7-average}. The next proposition
supplies the required estimate without any growth assumption on
the current.

\begin{proposition}\label{lattice averaging convergence}
Let $D=G_\R/V$ be non-classical with $G_\R$ simple, and let
$\Gamma\subset G_\R$ be a lattice, not necessarily cocompact
or torsion-free. For every $\Gamma$-invariant
closed positive $(1,1)$-current $\widetilde T$ on $D$, the average
\eqref{sec7-average} is well defined and is a smooth invariant
closed semipositive real $(1,1)$-form.
Moreover, $\mathcal A(\widetilde T)=0$ if and only if
$\widetilde T=0$.
\end{proposition}
\begin{proof}
Put $Y=\Gamma\backslash G_\R$ and let $\mu$ be its Haar
probability measure. 
As noted in the proof of Lemma \ref{sec7-averaging-properties}, $Y$ is smooth even if $\Gamma$
has torsion. 

We first define the average in each tangent direction,
allowing it to be infinite. Integration over compact cycles
implies the finiteness of the average in directions tangent to those cycles.
We then get the finiteness in directions corresponding to the simple roots by varying the base cycles, and finally in all directions by induction on the heights.

For $\chi\in C_c^\infty(Y)$, we define
$$
\mathcal A_\chi(\widetilde T)
=\int_Y\chi(g)g^*\widetilde T\,d\mu(g).
$$
Lemma~\ref{sec7-averaging-properties}(iii) implies that $\mathcal A_\chi(\widetilde T)$ is smooth and closed, and semipositive when $\chi\geq0$.

For fixed $x\in D$ and $v\in\T_x^{1,0}D$, the functional
$$
\chi\longmapsto
-\sqrt{-1}\mathcal A_\chi(\widetilde T)_x(v,\bar v)
$$
is linear and nonnegative when $\chi\geq0$. It therefore
defines a positive Radon measure $\nu_{x,v}$ on $Y$ by
\begin{equation}\label{defn of v_ov}
\int_Y\chi\,d\nu_{x,v}
=-\sqrt{-1}\mathcal A_\chi(\widetilde T)_x(v,\bar v),
\qquad \chi\in C_c^\infty(Y).
\end{equation}
When $\widetilde T$ is smooth, this measure has the density
$$
d\nu_{x,v}(g)
=-\sqrt{-1}\widetilde T_{gx}(g_*v,\overline{g_*v})\,d\mu(g).
$$
For a singular current, the pointwise value
$\widetilde T_{gx}$ need not be defined.
Instead, integration with the smooth kernels as in the proof
of Lemma~\ref{sec7-averaging-properties}(iii) gives the smooth
form $\mathcal A_\chi(\widetilde T)$.

Choose $\chi_j\in C_c^\infty(Y)$ with
$0\leq\chi_j\leq1$ and $\chi_j\uparrow1$.
By monotone convergence, the total mass is
$$
A_x(v):=\nu_{x,v}(Y)
=\lim_{j\to\infty}
\bigl[-\sqrt{-1}\mathcal A_{\chi_j}(\widetilde T)_x(v,\bar v)\bigr]
\in[0,\infty].
$$
We denote the set of directions with finite average by
$$
S_x=\{v\in\T_x^{1,0}D:A_x(v)<\infty\}.
$$

By letting
$j\to\infty$, the Cauchy--Schwarz inequality for the semipositive forms
$\mathcal A_{\chi_j}(\widetilde T)$ implies that
$$
A_x(v+w)\leq2A_x(v)+2A_x(w).
$$
Since $A_x(cv)=|c|^2A_x(v)$ for $c\ne0$ and $A_x(0)=0$,
the set $S_x$ is a complex subspace.
Right invariance of $\mu$ gives
$$
A_{gx}(g_*v)=A_x(v),\qquad g\in G_\R.
$$
Thus $S_o$ is $V$-invariant, and the spaces $S_x$ form the
smooth homogeneous subbundle $S=G_\R\times_V S_o$,
where $o=eV$.

We now show that the tangent spaces of the compact cycles lie in $S$.

Let $C=aZ_o\subset D$, $a\in G_\C$, be a small deformation of the base cycle,
set $d=\dim_\C C>0$, and choose a K\"ahler form $\kappa$
on $C$.
Consider the action map
$$
\rho:\, G_\R\times C\longrightarrow D,\qquad \rho(g,x)=gx.
$$
Transitivity of $G_\R$ on $D$ makes $\rho$ a smooth submersion, so its pullback
of a current is defined. Therefore the closed current $\rho^*\widetilde T$
is $\Gamma$-invariant in the first factor and descends to a
closed current $R_C$ on $Y\times C$.

Let $\operatorname{pr}_Y:Y\times C\to Y$ and
$\operatorname{pr}_C:Y\times C\to C$ be the projections.
The degree-zero current
$$
(\operatorname{pr}_Y)_*
\bigl(R_C\wedge\operatorname{pr}_C^*\kappa^{d-1}\bigr)
$$
is closed because $C$ is compact and $\kappa$ is closed.
It is therefore a finite constant $c_C$ on the connected manifold $Y$.
Since $x\mapsto gx$ is holomorphic and $\widetilde T$ is positive,
we have $c_C\geq0$. Thus, for $\chi\in C_c^\infty(Y)$,
$$
\begin{aligned}
\int_C\mathcal A_\chi(\widetilde T)\wedge\kappa^{d-1}
&=\int_{Y\times C}
R_C\wedge\operatorname{pr}_Y^*(\chi d\mu)
\wedge\operatorname{pr}_C^*\kappa^{d-1}\\
&=c_C\int_Y\chi\,d\mu.
\end{aligned}
$$
Here the integral over $Y\times C$ is understood in the
sense of currents.
For smooth $\widetilde T$, the first equality is Fubini's theorem.
It extends to currents by smooth regularization: the left-hand
side is continuous in $\widetilde T$,
and the right-hand side by continuity of pullback under a submersion.
The second equality follows from the definition of $c_C$.

Take $0\leq\chi_j\uparrow1$. The preceding identity and
monotone convergence give
$$
\int_C\lim_{j\to\infty}\operatorname{tr}_\kappa
\bigl(\mathcal A_{\chi_j}(\widetilde T)|_C\bigr)\,\kappa^d
=d\,c_C<\infty,
$$
where the trace is taken in a $\kappa$-orthonormal frame of
$\T^{1,0}C$. Thus the limit is finite almost everywhere on $C$.
Taking limits in the semipositivity estimate gives, for
$v\in\T_x^{1,0}C$,
$$
A_x(v)\leq |v|_\kappa^2\lim_{j\to\infty}
\operatorname{tr}_\kappa
\bigl(\mathcal A_{\chi_j}(\widetilde T)|_C\bigr)(x).
$$
It follows that $\T_x^{1,0}C\subset S_x$ almost everywhere.
The smooth bundle map
$$
\T^{1,0}C\longrightarrow(\T^{1,0}D/S)|_C
$$
therefore vanishes everywhere, so $\T^{1,0}C\subset S|_C$ for any
small deformation $C$ of the base cycle.

In particular, $A_o(e_\alpha)<\infty$ for every positive
compact root $\alpha$, where root vectors are understood modulo
$\mf b$. Roots in $\mf v$ give the zero tangent vector.

From Remark~\ref{full flag direct equivalence} and
Proposition~\ref{le D=0 conjecture}, we have that 
$$
\beta=\alpha-\gamma,\qquad
\alpha\in\Delta_+^{\mathrm c},\quad
\gamma\in\Delta_+^{\mathrm{nc}},
$$
for each noncompact simple root $\beta$.
Here $\alpha\notin\Delta(\mf v,\mf h)$, since otherwise
$[e_{-\gamma},e_\alpha]\in\mf b$ would imply $e_\beta\in\mf b$.
For small $t\in\C$, the cycle $\exp(te_{-\gamma})Z_o$ lies in
$D$ and passes through $o$. Thus
$\Ad(\exp(te_{-\gamma}))e_\alpha\bmod\mf b\in S_o$.
Since $S_o$ is a closed linear subspace, differentiating at
$t=0$ gives
$$
[e_{-\gamma},e_\alpha]\bmod\mf b\in S_o.
$$
This vector is a nonzero multiple of $e_\beta$.
Thus $A_o(e_\delta)<\infty$ for every simple positive root
$\delta$.

Normalize the root vectors by
$$
B_{\mf g}(e_\delta,e_{-\delta})=1,\qquad
\bar e_\delta=
\begin{cases}
-e_{-\delta},&\delta\in\Delta^{\mathrm c},\\
e_{-\delta},&\delta\in\Delta^{\mathrm{nc}}.
\end{cases}
$$
Let $E_\delta=e_\delta^L$ be the complex left invariant vector
field on $G_\R$ whose value at the identity is $e_\delta$.
It descends to $Y$, where we use the same notation.

Pull $\widetilde T$ back to $G_\R$ and descend it to a closed
current $\widehat T$ on $Y$.
Since $\widetilde T$ is positive, its coefficients are measures.
The same is true of its pullback $\widehat T$.
Identifying $\mu$ with its smooth Haar density, we define the
coefficient measures $\nu_{\delta\epsilon}$ by
$$
\int_Y\chi\,d\nu_{\delta\epsilon}
=\bigl\langle\widehat T(E_\delta,E_{-\epsilon}),\chi d\mu\bigr\rangle,
\qquad\chi\in C_c^\infty(Y),
$$
for $\delta,\epsilon\in\Delta_+$, where the
right-hand side pairs a distribution with a test density.

We write the measures on $Y$ as
$\nu_{\delta\epsilon}=\widehat T(E_\delta,E_{-\epsilon})\,\mu$;
when $\widehat T$ is smooth, the coefficient is an ordinary function.
The positive diagonal measures
$-\sqrt{-1}\widehat T(E_\delta,\overline{E_\delta})\,\mu$
are precisely $\nu_{o,e_\delta}$ as given in \eqref{defn of v_ov}, and
$|\nu_{\delta\delta}|=\nu_{o,e_\delta}$.
More generally, the matrix of measures
$$
\bigl(-\sqrt{-1}\nu_{\delta\epsilon}\bigr)_{\delta,\epsilon}
$$
is Hermitian and positive semidefinite on every relatively
compact Borel set of $Y$. 

Choose a complete left invariant Riemannian metric on $G_\R$;
the induced metric on $Y$ is also complete. Functions of the
distance to a fixed point give compactly supported Lipschitz
cutoffs with differential bounded by $C/R$.
Smoothing them by convolution with a fixed nonnegative, compactly
supported kernel of integral one for the right action of $G_\R$
preserves this bound up to a constant. Squaring the resulting
functions gives $\chi_R\in C_c^\infty(Y)$ with
$$
0\leq\chi_R\leq1,\qquad \chi_R\longrightarrow1,\qquad
|E\chi_R|\leq\frac{C_E}{R}\sqrt{\chi_R}
$$
for every fixed left invariant field $E$. These fields have
bounded length on $Y$, and in the convolution their derivatives
are replaced by those along $\Ad(g)E$, with $g$ in a fixed
compact set.

Now we prove $A_o(e_\alpha)<\infty$ for all positive roots $\alpha$ by induction on the heights. The simple roots have already been treated.
If $\alpha$ is not simple, write $\alpha=\beta+\gamma$ with
$\beta,\gamma$ positive of smaller height.
By invariance of the Killing form,
$$
[e_\beta,e_\gamma]=ce_\alpha,\qquad
[e_\gamma,e_{-\alpha}]=ce_{-\beta},\qquad
[e_{-\alpha},e_\beta]=ce_{-\gamma}
$$
for the same nonzero constant $c$.
The projections of $E_\beta$ and $E_\gamma$ have type $(1,0)$
or vanish, so $\widehat T(E_\beta,E_\gamma)=0$. Expanding
$d\widehat T(E_\beta,E_\gamma,E_{-\alpha})=0$ gives
$$
c(\nu_{\alpha\alpha}-\nu_{\beta\beta}-\nu_{\gamma\gamma})
=E_\beta\nu_{\gamma\alpha}-E_\gamma\nu_{\beta\alpha}
$$
as an identity of distributions. The fields $E_\delta$ have
zero divergence with respect to $\mu$, since their real and
imaginary parts generate right translations. Integrating by
parts against $\chi_R$ and applying the Cauchy--Schwarz
inequality to the coefficient measures $\nu_{\delta\epsilon}$, we obtain
$$
\begin{aligned}
|\langle E_\beta\nu_{\gamma\alpha},\chi_R\rangle|
&\leq\frac{C_\beta}{R}
\int_Y\sqrt{\chi_R}\,d|\nu_{\gamma\alpha}|\\
&\leq\frac{C_\beta}{R}
\sqrt{A_o(e_\gamma)}
\left(\int_Y\chi_R\,d\nu_{o,e_\alpha}\right)^{1/2},
\end{aligned}
$$
and the analogous estimate with $\beta,\gamma$ interchanged.
Taking absolute values in the closedness identity gives
$$
\begin{aligned}
\int_Y\chi_R\,d\nu_{o,e_\alpha}
\leq{}& A_o(e_\beta)+A_o(e_\gamma)\\
&+\frac{C_{\beta,\gamma}}{R}
\left(\int_Y\chi_R\,d\nu_{o,e_\alpha}\right)^{1/2}
\bigl(\sqrt{A_o(e_\beta)}+\sqrt{A_o(e_\gamma)}\bigr).
\end{aligned}
$$
If $M_R$ denotes the integral on the left, this has the form
$M_R\leq a+bR^{-1}\sqrt{M_R}$ with $a,b<\infty$ by the induction
hypothesis. Hence $\sqrt{M_R}\leq b+\sqrt a$ for $R\geq1$.
Fatou's lemma now gives $A_o(e_\alpha)<\infty$, completing
the induction.

We have proved $S=\T^{1,0}D$. Polarization of $A_x$ now
defines a semipositive $(1,1)$-form $\eta$ with
$-\sqrt{-1}\eta_x(v,\bar v)=A_x(v)$.
The identity $A_{ax}(a_*v)=A_x(v)$ makes $\eta$ invariant,
and hence smooth.
For the nonnegative weights $\chi_j\uparrow1$ chosen above,
we have
$$
0\leq\mathcal A_{\chi_j}(\widetilde T)\leq\eta,\qquad
\mathcal A_{\chi_j}(\widetilde T)_x\longrightarrow\eta_x.
$$
The Cauchy--Schwarz inequality bounds their mixed coefficients
on compact subsets. Dominated convergence therefore gives
convergence as currents. 
Since $d\mathcal A_{\chi_j}(\widetilde T)=0$ for every $j$,
we have $d\eta=0$.

It remains to identify $\eta$ with \eqref{sec7-average}.
Since every compactly supported smooth $(n-1,n-1)$-form
is dominated by a compactly supported smooth strongly positive
form, monotone convergence and
$\mathcal A_{\chi_j}(\widetilde T)\to\eta$ show that the weak
integral is absolutely convergent.
By dominated convergence, for every such test form $\varphi$,
$$
\int_Y\langle g^*\widetilde T,\varphi\rangle\,d\mu(g)
=\langle\eta,\varphi\rangle.
$$
Thus $\mathcal A(\widetilde T)=\eta$.
By the positivity argument in
Lemma~\ref{sec7-averaging-properties}(ii), this average
is zero only if $\widetilde T=0$.
\end{proof}

On a possibly singular complex space, we work with closed
positive $(1,1)$-currents given locally by plurisubharmonic
potentials. Such a current $T$ is a \emph{K\"ahler current}
if every point has a neighborhood on which $T\geq\varepsilon\omega$
for some $\varepsilon>0$ and some smooth Hermitian form $\omega$,
induced from a local embedding in a complex manifold.
This is the local form of the positivity condition in
\cite[Definition~13.15]{Demailly10}; the constant $\varepsilon$
may depend on the neighborhood.

\begin{theorem}\label{vanishing of any line}
\label{lattice positivity obstruction}
\label{lattice quotients not Kahler}
\label{lattice smooth semipositive vanishing}
Let $D=G_\R/V$ be non-classical with $G_\R$ simple, and let
$\Gamma\subset G_\R$ be a lattice.
Put $X=\Gamma\backslash D$. Then:
\begin{enumerate}
\item[(i)] the quotient $X$ admits no K\"ahler current;
\item[(ii)] when $X$ is smooth, every closed positive
$(1,1)$-current on $X$ is zero.
\end{enumerate}
\end{theorem}

\begin{proof}
By Proposition~\ref{lattice averaging convergence} and
Lemma~\ref{sec7-invariant-semipositive}, every $\Gamma$-invariant
closed positive $(1,1)$-current on $D$ is zero.
Let $q:D\to X$ be the quotient map. A K\"ahler current on $X$
pulls back, using its local plurisubharmonic potentials, to a
closed positive $(1,1)$-current on $D$. On the dense open set
where the effective quotient map is a local biholomorphism,
this pullback is strictly positive. It is therefore nonzero,
which proves (i).

If $X$ is smooth, the same pullback is defined for every closed
positive $(1,1)$-current $T$ on $X$. It is injective even when
$q$ is branched. Indeed, locally $q$ is a finite quotient map
of degree $m$, and $q_*q^*T=mT$; this identity follows first
for smooth local potentials and then by regularization.
Thus $q^*T=0$ implies $T=0$, proving (ii).
\end{proof}

\begin{remark}\label{mixed positive currents}\label{mixed smooth currents}
In the semisimple setting of $G_\R$, let $D=G_\R/V$ be
non-classical and let $\Gamma\subset G_\R$ be cocompact and
torsion-free. Then $X=\Gamma\backslash D$ admits no K\"ahler
current.

Indeed, averaging the lift of such a current would give a
$G_\R$-invariant K\"ahler form on $D$
by Lemma~\ref{sec7-averaging-properties}.
The root decomposition splits the tangent space and the
invariant complex structure according to the simple factors.
Since $D$ is non-classical, at least one of the corresponding
flag domains is non-classical. Restricting the averaged form
to this factor contradicts
Lemma~\ref{sec7-invariant-semipositive}.
\end{remark}

\begin{corollary}\label{positive current openness}
Let $D=G_\R/V$ be non-classical with $G_\R$ simple, and let
$\Gamma\subset G_\R$ be cocompact and torsion-free.
Every sufficiently small holomorphic deformation of
$X=\Gamma\backslash D$ admits no nonzero closed positive
$(1,1)$-current.
\end{corollary}

\begin{proof}
Consider a small deformation $\{X_t\}_{t\in \Delta}$ with $X_0=X$.
Identify the nearby fibers smoothly with $X$, and choose
a smooth family of Hermitian forms $h_t$.
Write $n=\dim_\C X$.

If the assertion fails, there are parameters $t_\nu\to0$
and nonzero closed positive $(1,1)$-currents $T_\nu$ on
$X_{t_\nu}$. Normalize them so that
$$
\int_X T_\nu\wedge h_{t_\nu}^{n-1}=1.
$$
Their masses with respect to a fixed background metric are
uniformly bounded. After passing to a subsequence, they
converge weakly to a closed current $T$, which is positive and
of type $(1,1)$ on $X$.
The smooth convergence $h_{t_\nu}\to h_0$ gives
$$
\int_X T\wedge h_0^{n-1}=1.
$$
Thus $T\ne0$, contradicting
Theorem~\ref{lattice positivity obstruction}(ii).
\end{proof}

\section{Nonalgebraicity of flag domain quotients}
\label{applications}

We apply the vanishing theorem to line bundles and meromorphic
maps on lattice quotients. We then show that these quotients
cannot be open subsets of compact spaces in Fujiki class
$\mathcal C$, even when the complement is not analytic.

\subsection{Pseudo-effective line bundles and meromorphic maps}
\label{finite volume Kahler obstruction}

A holomorphic line bundle is pseudo-effective if it admits a
singular Hermitian metric with locally integrable weights and
positive curvature current.

\begin{corollary}\label{sec7-meromorphic-consequences}
\label{lattice maps to Kahler}
Let $X=\Gamma\backslash D$ be a smooth lattice quotient as in
Theorem~\ref{lattice positivity obstruction}. Then:
\begin{enumerate}
\item[(1)] For a holomorphic line bundle on $X$,
pseudo-effectivity is equivalent to unitary flatness and to
the existence of a smooth Hermitian metric with semipositive
curvature.

\item[(2)] The manifold $X$ has no nonzero effective divisors,
and $H^0(X,L_X)=0$ for every holomorphic line bundle
$L_X\not\simeq\mathcal O_X$ on $X$.

\item[(3)] Every meromorphic map from $X$ to a K\"ahler manifold
or a quasi-projective variety is constant.
In particular, $a(X)=0$ if $X$ is compact.
\end{enumerate}
\end{corollary}
\begin{proof}
For (1), let $h$ be a singular Hermitian metric with positive
curvature current. Its curvature is zero by
Theorem~\ref{lattice positivity obstruction}(ii). Thus each
local weight $\varphi$ satisfies
$$
\partial\bar\partial\varphi=0
$$
in the sense of distributions and has a pluriharmonic
representative. These representatives give a smooth Hermitian
metric whose Chern connection is unitary and flat. Conversely,
a unitary flat metric is smooth and has zero curvature.

For (2), the same vanishing theorem gives $[E]=0$ for every
effective divisor $E$, so $E=0$. A nonzero section of a line
bundle is therefore nowhere vanishing and trivializes the bundle.

For (3), let $X\dashrightarrow M$ be a meromorphic map to a
K\"ahler manifold with K\"ahler form $\omega$. Resolve its graph to obtain
a proper modification $\varpi:\widetilde X\to X$ and a
holomorphic map $f:\widetilde X\to M$.
Since $\varpi$ is proper, $\varpi_*f^*\omega$ is a closed
positive $(1,1)$-current on $X$, and
Theorem~\ref{lattice positivity obstruction}(ii) gives
$$
\varpi_*f^*\omega=0.
$$
It follows that $f^*\omega=0$ where $\varpi$ is an isomorphism,
and hence everywhere by continuity. Since $\omega$ is positive
definite, we conclude that $df=0$, and $f$ is constant.
For a quasi-projective target, we can compose with an embedding into
projective space.
\end{proof}

This proves the implications from non-classicality in
Corollary~\ref{compact quotient criteria}(c) and (d).
For compact quotients by torsion-free lattices, the same
conclusions hold under sufficiently small holomorphic deformations
by Corollary~\ref{positive current openness}; these deformations
also lie outside Fujiki class $\mathcal C$ by \cite{DP}.

\begin{remark}\label{positive currents on modifications}
Let $X$ be as in Corollary~\ref{sec7-meromorphic-consequences},
and let $\varpi:\widetilde X\to X$ be a proper modification
with $\widetilde X$ smooth. Then
$$
\mathcal P_d^{1,1}(\widetilde X)
=\left\{\sum_j a_j[E_j]:a_j\geq0\right\},
$$
where $E_j$ are the exceptional prime divisors of $\varpi$ and
the sum is locally finite.
Indeed, $\varpi_*T=0$ for every
$T\in\mathcal P_d^{1,1}(\widetilde X)$. Hence $T$ vanishes
where $\varpi$ is an isomorphism and is supported on the
exceptional set. The assertion follows from the support
theorem for currents \cite[Chapter~III, Corollaries~2.11 and~2.14]{DemaillyCADG}.
In particular, every smooth closed semipositive $(1,1)$-form on
$\widetilde X$ is zero.
\end{remark}

\subsection{Open embeddings and Fujiki class \texorpdfstring{$\mathcal C$}{C}}
\label{open quotient applications}
We begin with the extension of the nonalgebraicity theorem of
\cite{GRT} to Fujiki class $\mathcal C$, observed by Koll\'ar
and recorded in the introduction to \cite{GRT}. Here a Zariski open subset
means the complement of a closed complex analytic subset.
The proof uses rational chains and does not require $\Gamma$
to be a lattice. For lattices quotients, the vanishing theorem will allow
us to remove the condition that the boundary be analytic.

\begin{theorem}\label{no Fujiki compactification}
\label{torsion Fujiki compactification}
Let $D=G_\R/V$ be a non-classical flag domain with $G_\R$
simple, and let $\Gamma\subset G_\R$ be an infinite,
finitely generated discrete subgroup. Assume that either
$\Gamma$ is torsion-free or $G_\R$ is of adjoint type.
Then the normal complex space $X=\Gamma\backslash D$ is not
biholomorphic to a Zariski open subset of a compact complex
space in Fujiki class $\mathcal C$.
\end{theorem}

\begin{proof}
Suppose that $X$ is Zariski open in a compact complex space
$\bar X\in\mathcal C$. Replacing $\bar X$ by the normalization
of the irreducible component containing $X$, we may assume
that $\bar X$ is normal.

We first reduce to the torsion-free case. If $\Gamma$ has
torsion and $G_\R$ is of adjoint type, then $\Gamma$ is linear.
Selberg's lemma gives a torsion-free subgroup of finite index
in $\Gamma$. Intersecting its finitely many conjugates, we obtain
a torsion-free normal subgroup $\Gamma_0\subset\Gamma$ of
finite index. The finite map
$$
f:X_0=\Gamma_0\backslash D\longrightarrow X
$$
is unramified over the regular part $X_{\mathrm{reg}}$ of $X$
\cite[Lemma~5.1 and \S5.1]{GRT}.
Set
$$
A=(\bar X\setminus X)\cup\operatorname{Sing}(\bar X).
$$
Then $\bar X\setminus A=X_{\mathrm{reg}}$, and normality
ensures that the proper analytic subset $A$ does not locally
separate $\bar X$. By the Grauert--Remmert extension theorem
\cite[Theorem~3.4]{DG94}
(see also \cite[Theorem~4.1]{Lavoine15}), the cover over
$X_{\mathrm{reg}}$ extends uniquely to a finite surjective
map $\bar f:\bar X_0\to\bar X$ with $\bar X_0$ normal.
Uniqueness identifies $\bar f^{-1}(X)$ with $X_0$.

To see that $\bar X_0\in\mathcal C$, choose a smooth compact
K\"ahler modification $Y\to\bar X$, and let $Y_0$ be the
normalization of the component of
$Y\times_{\bar X}\bar X_0$ dominating $\bar X_0$.
The map $Y_0\to Y$ is finite, while
$Y_0\to\bar X_0$ is a modification. Hence $Y_0$ is a compact
K\"ahler space
\cite[Chap.~II, Proposition~1.3.1(v),(vi)]{Varouchas},
and $\bar X_0\in\mathcal C$.
Since $\Gamma_0$ is still infinite and finitely generated,
we may replace $(\Gamma,X,\bar X)$ by
$(\Gamma_0,X_0,\bar X_0)$ and assume that $\Gamma$ is
torsion-free.

We now construct a smooth quasi-projective modification of $X$.
Resolve $\bar X$ without changing $X$, and retain the notation
$\bar X$ for the resulting smooth compactification.
Choose a modification
$\varpi:\widetilde X\to\bar X$ with $\widetilde X$ compact
K\"ahler. By Corollary~\ref{subbundle restriction and forms},
every holomorphic two-form on $\bar X$ vanishes on $X$,
and hence vanishes on $\bar X$. Holomorphic forms are unchanged under
modifications of smooth compact complex manifolds, and hence
$$
H^0(\widetilde X,\Omega_{\widetilde X}^2)
\simeq H^0(\bar X,\Omega_{\bar X}^2)=0.
$$
Hodge decomposition therefore gives
$H^2(\widetilde X,\R)=H^{1,1}(\widetilde X,\R)$.
The K\"ahler cone is open, and rational classes are dense
in $H^2(\widetilde X,\R)$, so it contains a rational class.
Kodaira's embedding theorem then implies that
$\widetilde X$ is projective. Its analytic boundary
$\widetilde X\setminus\varpi^{-1}(X)$ is algebraic by Chow's
theorem. Thus $M=\varpi^{-1}(X)$ is smooth and
quasi-projective.

We next show that any two points of $M$ can be joined by
a chain of complete rational curves. This holds for $X$
by Theorem~\ref{equ of non-classical D}.
Embed $\widetilde X$ in projective space and represent the
meromorphic inverse of $\varpi$ by holomorphic sections
of a line bundle on $\bar X$. Principalizing their base
ideal \cite[Theorem~1.10]{BM} gives a finite sequence of
blowups along smooth centers
$\rho:\widehat X\to\bar X$.
Dividing the pulled-back sections by a local generator
of the resulting invertible ideal defines a holomorphic
map $h:\widehat X\to\widetilde X$ satisfying
$$
\varpi\circ h=\rho.
$$
Its image is closed and contains a dense open subset of
$\widetilde X$, so $h$ is surjective.

Consider a blowup of a smooth manifold $M_0$ along a smooth
center $S$. Every map $f:\mathbb P^1\to M_0$ lifts to the
blowup. If its image is not contained in $S$, the ideal
$f^{-1}\mathcal I_S\cdot\mathcal O_{\mathbb P^1}$ is
invertible, and the universal property gives a lift.
If its image lies in $S$, a line summand in the
Birkhoff--Grothendieck splitting of $f^*N_{S/M_0}$ gives a
section of $\mathbb P(f^*N_{S/M_0})$, and hence a lift to
the exceptional divisor. The fibers of the blowup are points or projective
spaces. At a junction of two curves, the chosen lifts may
have different endpoints, but these lie in the same fiber
and can be joined by a line. The same argument joins the
endpoints of the lifted chain to any prescribed points in
their fibers. Applying this to each blowup in $\rho$ shows
that any two points of $\rho^{-1}(X)$ can be joined by a
chain of complete rational curves. Since
$h(\rho^{-1}(X))=M$, the same holds for $M$.

Let $F$ be a very general fiber of the Shafarevich map on
the open subset $M^0\subset M$ supplied by
\cite[Theorem~5.3]{GRT}. The fiber $F$ is closed in $M$.
By part~(d) of that theorem and \cite[Remark~5.4]{GRT},
every complete rational curve meeting $F$ is contained
in $F$: its normalization is $\mathbb P^1$, whose
fundamental group is trivial.
Every rational chain beginning in $F$ therefore remains
in $F$. Since such chains reach every point of $M$, we
have $F=M$. The defining property of the Shafarevich map
then implies that $\pi_1(M)$ is finite.

Finally, $\varpi$ is an isomorphism away from an analytic
subset of $\bar X$ of codimension at least two. Every loop
in $X$, based outside this subset, can be deformed to
avoid it and hence lifted to $M$. Thus
$\pi_1(M)\to\pi_1(X)$ is surjective. Since $D$ is simply
connected and $\Gamma$ acts freely,
we conclude that $\pi_1(X)\simeq\Gamma$ is infinite, a contradiction.
\end{proof}

Now we return to the application of the main results of our paper to the lattices quotients of non-classical flag domains.

\begin{theorem}\label{lattice open embeddings}
\label{no proper Kahler source}
Let $D=G_\R/V$ be non-classical with $G_\R$ simple, and let
$\Gamma\subset G_\R$ be a lattice. Then $X=\Gamma\backslash D$
is not biholomorphic to an open subset of a compact complex
space in Fujiki class $\mathcal C$.
If $X$ is smooth, there is no proper surjective holomorphic
map from a K\"ahler manifold to $X$.
\end{theorem}
\begin{proof}
Suppose first that $X$ is smooth and $f:M\to X$ is a proper
surjective holomorphic map from a K\"ahler manifold.

Let $\omega$ be a
K\"ahler form on $M$, the proper map $f$ defines a closed positive
$(1,1)$-current $f_*(\omega^{r+1})$, where $r=\dim M-\dim X$. 
On a nonempty open set $U$
of regular values of $f$, integration along the compact fibers $f^{-1}(x)$, $x\in U$, gives a strictly positive form. The current $f_*(\omega^{r+1})$ is therefore nonzero,
contrary to
Theorem~\ref{lattice positivity obstruction}(ii).

For an arbitrary lattice, suppose that $X$ is open in
$\bar X\in\mathcal C$, and choose a compact K\"ahler
modification $\varpi:\widetilde X\to\bar X$ with K\"ahler
form $\omega$. Let $W\subset D\times\widetilde X$ be the
closure of the graph of the meromorphic map
$\varpi^{-1}\circ\pi$, where $\pi:D\to X$ is the quotient
map. Let $q:D\times\widetilde X\to D$ and
$p:D\times\widetilde X\to\widetilde X$ be the projections.
Since $\widetilde X$ is compact, we have that $q$ is proper, and
$$
T=q_*(p^*\omega\wedge[W])
$$
is a closed positive $(1,1)$-current on $D$.
The identity $\pi\circ\gamma=\pi$ for $\gamma\in\Gamma$
shows that $W$ is preserved by the action on the first factor,
so $T$ is $\Gamma$-invariant. 

On a nonempty open subset $U\subset D$, the map
$f=\varpi^{-1}\circ\pi$ is well defined and locally
biholomorphic. Hence we have that
$$
T|_U=f^*\omega,
$$
is strictly
positive, and $T\ne0$. However, the proof of
Theorem~\ref{lattice positivity obstruction} shows that
every $\Gamma$-invariant closed positive $(1,1)$-current
on $D$ vanishes. This is a contradiction.
\end{proof}


\section{Invariant complex structures and algebraic quotients}\label{ppl}
We work in the semisimple setting of Section~\ref{Pre}.
When $G_\R$ is of Hermitian type, the smooth manifold underlying
$D$ also carries a classical invariant complex structure.
We use this complex structure to obtain projective and
quasi-projective quotients on the same underlying smooth
manifolds as the non-classical quotients above.

\subsection{Classical complex structures}

Let $G_c\subset G_\C$ be the compact real form with Lie algebra
$\mf g_c$.

\begin{proposition}\label{equi complex}
The $G_\R$-invariant complex structures on $D=G_\R/V$ correspond to
complex subspaces $\mf s\subseteq\mf n_+\oplus\mf n_-$ satisfying
\begin{align}
\ad\mf v(\mf s)&\subseteq\mf s,\label{associatedD}\\
\mf s\oplus\bar{\mf s}&=\mf n_+\oplus\mf n_-,\label{almost complex D}\\
[\mf s,\mf s]&\subseteq\mf s\oplus\mf v.\label{integrable D}
\end{align}
The corresponding tangent bundle is
\begin{equation}\label{complex D}
\T^{\mf s}D=G_\R\times_V\mf s\subset\T^\C D.
\end{equation}
Each such subspace also satisfies $[\mf s,\mf s]\subseteq\mf s$.
There is a parabolic subgroup $B^{\mf s}\subset G_\C$ with Lie algebra
$\mf b^{\mf s}=\bar{\mf s}\oplus\mf v$ and Levi subgroup $V_\C$.
With the corresponding complex structure, $G_c/V$ is
biholomorphic to $G_\C/B^{\mf s}$. The $G_\R$-orbit of the
base point is open and is biholomorphic to $D$ with the complex
structure \eqref{complex D}.
Equivalently, these subspaces are
$$
\mf s=\bigoplus_{\alpha\in\Delta_+^{\mf s}\setminus
\Delta(\mf v,\mf h)}\mf g_\alpha,
$$
where $\Delta_+^{\mf s}$ is a positive root system for which
$\Delta(\mf v,\mf h)$ is generated by a subset of the simple roots.
\end{proposition}
\begin{proof}
Since $V$ is connected, \eqref{associatedD} says that $\mf s$
is $V$-invariant, so it defines the bundle $G_\R\times_V\mf s$.
Equation \eqref{almost complex D} gives an almost-complex
structure, and \eqref{integrable D} makes it integrable by the
Newlander--Nirenberg theorem \cite{NN}.
Since $\mf h\subset\mf v$, the subspace $\mf s$ is a sum of root
spaces. Conjugation with respect to either $\mf g_0$ or the compact
real form $\mf g_c=\mf k_0\oplus\sqrt{-1}\mf p_0$ sends
$\mf g_\alpha$ to $\mf g_{-\alpha}$, so the same conditions define
an invariant complex structure on $G_c/V$.

The subspace $\mf b^{\mf s}=\bar{\mf s}\oplus\mf v$ is a Lie
subalgebra by \eqref{associatedD}--\eqref{integrable D}.
Its normalizer in $G_\C$ is a closed algebraic subgroup
$B^{\mf s}$ with Lie algebra $\mf b^{\mf s}$. Indeed, write
an element of the Lie algebra of the normalizer as $Y+Z$,
with $Y\in\mf b^{\mf s}$ and $Z\in\mf s$. Then
$$
[Z,\mf h]\subseteq\mf s\cap\mf b^{\mf s}=0.
$$
The centralizer of $\mf h$ is $\mf h$, and $\mf s\cap\mf h=0$, so
$Z=0$.
The $G_c$-orbit of the base point in $G_\C/B^{\mf s}$ is
compact. Since conjugation with respect to $\mf g_c$ sends
$\mf s$ to $\bar{\mf s}$, we have
$\mf g_c+\mf b^{\mf s}=\mf g$ as real vector spaces.
The orbit is therefore also open. Hence $G_c$ acts transitively
and $G_\C/B^{\mf s}$ is compact, so $B^{\mf s}$ is parabolic
\cite[Proposition~1.4.10]{FHW}.
Its intersection with $G_c$ is the centralizer of a torus in $G_c$,
hence connected, with Lie algebra $\mf g_c\cap\mf b^{\mf s}=\mf v_0$; so
it equals $V$.

The parabolics $\mf b^{\mf s}$ and $\bar{\mf b}^{\mf s}$ are opposite:
$\mf g=\mf s\oplus\bar{\mf s}\oplus\mf v$ gives
$\mf b^{\mf s}\cap\bar{\mf b}^{\mf s}=\mf v$, and no root space of
$\bar{\mf s}$ has its negative in $\mf b^{\mf s}$. Hence their common
Levi subalgebra is $\mf v$ and $\mf s$ is the opposite nilradical, which
gives the root description and $[\mf s,\mf s]\subseteq\mf s$.
Let $\sigma$ denote conjugation with respect to $\mf g_0$.
It sends $\mf g_\alpha$ to $\mf g_{-\alpha}$ and preserves
$\mf v$. Thus
$\sigma(\mf b^{\mf s})=\mf s\oplus\mf v$, and
$\sigma(B^{\mf s})$ is the opposite parabolic. Their
intersection is the common Levi subgroup, so
$$
B^{\mf s}\cap\sigma(B^{\mf s})=V_\C=B\cap\sigma(B).
$$
Since $\sigma$ fixes $G_\R$ pointwise and $V_\C\subset B$, this gives
$G_\R\cap B^{\mf s}\subseteq G_\R\cap V_\C\subseteq G_\R\cap B=V$; the
reverse inclusion follows because $V$ is connected with
$\mf v\subseteq\mf b^{\mf s}$. Hence $G_\R\cap B^{\mf s}=V$.
Applying the same argument to $\sigma$ gives
$\mf g_0+\mf b^{\mf s}=\mf g$, so the $G_\R$-orbit is open.
Its holomorphic tangent space at the base point is
$\mf g/\mf b^{\mf s}\simeq\mf s$.
This identifies $D$, with the specified complex structure, with
that open orbit. Conversely, the displayed root subspaces
satisfy \eqref{associatedD}--\eqref{integrable D}.
\end{proof}

For the given embedding $D\subset\check D=G_\C/B$, take $\mf s=\mf n_-$,
so that $\mf b^{\mf n_-}=\mf n_+\oplus\mf v=\mf b$; then
\eqref{complex D} recovers the given structure, and
$$
\T^{1,0}D=G_\R\times_V\mf n_-
\simeq\T\check D|_D.
$$

Assume now that $G_\R/K$ is Hermitian symmetric, and fix an
invariant complex structure on it. Its holomorphic tangent
space at the base point is a $\mf k$-invariant subspace
$\mf p_-'\subset\mf p$, with
$\T^{1,0}(G_\R/K)=G_\R\times_K\mf p_-'$ and
$$
\mf p=\mf p_-'\oplus\mf p_+',\qquad
\mf p_+'=\overline{\mf p_-'},\qquad
[\mf p_-',\mf p_-']=0.
$$

\begin{theorem}\label{new complex D}\label{check DR}
Let $D=G_\R/V$ be non-classical and $G_\R$ of Hermitian type.
Then
\begin{equation}\label{new complex bundle}
\T'D=G_\R\times_V(\mf k_-\oplus\mf p_-')
\end{equation}
defines a different invariant complex structure on the same
smooth manifold, agreeing with the original one on the fibers $K/V$.
Denote the resulting complex manifold by $D'$.
The projection $p:D'\to G_\R/K$ is holomorphic, and $D'$ is the
classical open orbit in $\check D'=G_\C/B'$, where
$$
\operatorname{Lie}(B')=\mf k_+\oplus\mf p_+'\oplus\mf v.
$$
\end{theorem}
\begin{proof}
Put $\mf n_-'=\mf k_-\oplus\mf p_-'$.
It is $\mf v$-invariant and complementary to its conjugate in
$\mf n_+\oplus\mf n_-$. Moreover,
$$
[\mf k_-,\mf k_-]\subseteq\mf k_-,\qquad
[\mf k_-,\mf p_-']\subseteq\mf p_-',\qquad
[\mf p_-',\mf p_-']=0.
$$
Thus $[\mf n_-',\mf n_-']\subseteq\mf n_-'$, and
Proposition~\ref{equi complex} gives the complex structure
and identifies $D'$ with an open orbit in $\check D'$.
The differential of $p$ sends
$\mf n_-'$ onto $\mf p_-'$, so $p$ is holomorphic.
The original structure has tangent space
$\mf n_-=\mf k_-\oplus\mf p_-$.
Since $D$ is non-classical, $\mf p_-\ne\mf p_-'$; otherwise $p$
would be holomorphic for the original structure as well.
The complex structure on each fiber $K/V$ is unchanged,
since its holomorphic tangent space is still $\mf k_-$.
\end{proof}


\subsection{Projective and quasi-projective quotients}\label{compact quotient applications}
For compact quotients, the classical complex structure gives
projectivity, while the original non-classical structure does
not belong to Fujiki class $\mathcal C$. We then consider
arithmetic quotients, which need not be compact.

\begin{theorem}\label{comact X not ddbar}
Let $D=G_\R/V$ be a non-classical flag domain.
Let $\Gamma\subset G_\R$ be a cocompact torsion-free discrete
subgroup. Then $X=\Gamma\backslash D$ does not belong to Fujiki
class $\mathcal C$.
\end{theorem}
\begin{proof}
By Remark~\ref{mixed positive currents}, $X$ admits no K\"ahler
current. Thus $X\notin\mathcal C$ by \cite{DP}.
\end{proof}

\begin{theorem}\label{proj no Fujiki}
Let $D=G_\R/V$ be non-classical with $G_\R$ of Hermitian type,
and let $\Gamma\subset G_\R$ be a cocompact torsion-free discrete
subgroup. The compact complex manifolds $X=\Gamma\backslash D$ and
$X'=\Gamma\backslash D'$ have the same underlying smooth manifold,
with $X\notin\mathcal C$ and $X'$ projective.
\end{theorem}
\begin{proof}
The two invariant complex structures descend to the same smooth
quotient, and Theorem~\ref{comact X not ddbar} applies to $X$.

Since $G_\R/K$ is Hermitian symmetric and the projection
$D'\to G_\R/K$ is holomorphic, \cite[(4.22)]{GS} gives a
positive holomorphic vector bundle $F$ on the compact quotient
$X'$. The induced metric on $\det F$ has positive curvature,
so $X'$ is projective by Kodaira's embedding theorem.
\end{proof}

\begin{corollary}\label{arithmetic complex structures}
Let $D=G_\R/V$ be non-classical with $G_\R$ simple, linear
and of Hermitian type, and let $\Gamma\subset G_\R$ be a neat
arithmetic subgroup. Then $X=\Gamma\backslash D$ and
$X'=\Gamma\backslash D'$ have the same underlying smooth
manifold, with
$$
\mathcal P_d^{1,1}(X)=\{0\},\qquad X'\text{ quasi-projective},
$$
where $\mathcal P_d^{1,1}(X)$ denotes the cone of closed positive
$(1,1)$-currents on $X$.
Moreover, the projection
$$
\tilde p:X'\longrightarrow B=\Gamma\backslash G_\R/K
$$
is a projective algebraic morphism.
\end{corollary}
\begin{proof}
By Baily--Borel \cite{BB66},
the quotient $B=\Gamma\backslash G_\R/K$ is a smooth quasi-projective
variety. 
We first show that the flag bundle $X'\to B$ is algebraic when $\Gamma$ is a congruence subgroup.

Let $P\subset G_\C$ define the compact dual of $G_\R/K$.
Its Levi factor is $K_\C$, and the parabolic $B'$ in
Theorem~\ref{new complex D} is contained in $P$. Hence
$P/B'\simeq K/V$. For $m$ sufficiently large,
$K_{K/V}^{-m}$ is very ample, and its natural
$K_\C$-linearization gives an equivariant embedding
$$
K/V\longrightarrow\mathbb P(W),\qquad
W=H^0(K/V,K_{K/V}^{-m})^*.
$$

The $K_\C$-action on $W$ extends to $P$ with trivial
action of the unipotent radical. We give $G_\R\times_K W$
the holomorphic structure induced by the restriction of
$G_\C\times_P W$ to $G_\R/K$. By
\cite[Propositions~4.2--4.3]{Milne88}, there is an algebraic
vector bundle $\mathcal E$ on $B$ with
$$
\mathcal E^{\mathrm{an}}
\simeq\Gamma\backslash(G_\R\times_K W).
$$

For each $d$, let $I_d\subset\operatorname{Sym}^d W^*$
be the degree-$d$ part of the homogeneous ideal of
$K/V\subset\mathbb P(W)$. Since the embedding is equivariant,
$I_d$ is a $P$-submodule. The inclusion
$I_d\hookrightarrow\operatorname{Sym}^d W^*$ gives an
algebraic bundle map
$$
\mathcal I_d\longrightarrow\operatorname{Sym}^d\mathcal E^*,
$$
where $\mathcal I_d$ is the automorphic bundle associated
with $I_d$. On each fiber of $\mathbb P(\mathcal E)\to B$,
this map gives the degree-$d$ equations of
$K/V\subset\mathbb P(W)$.
For $N$ large enough, $I_1,\ldots,I_N$ generate the
homogeneous ideal of $K/V$. In a local holomorphic
trivialization over $U\subset B^{\mathrm{an}}$, the corresponding
equations define $U\times(K/V)$ in $U\times\mathbb P(W)$.
These products glue by the transition maps of the flag bundle.
Thus $X'$ is the analytification of the closed subscheme of
$\mathbb P(\mathcal E)$ defined by these equations.
Hence we have that $X'\to B$ is projective, and $X'$ is quasi-projective.

Finally, the two invariant complex structures descend to
the same smooth quotient. Since an arithmetic subgroup is
a lattice and neatness implies freeness,
Theorem~\ref{lattice positivity obstruction}(ii) gives
$\mathcal P_d^{1,1}(X)=\{0\}$.
\end{proof}

Assume again that $\Gamma$ is cocompact and torsion-free.
When $G_\R$ is simple,
the vanishing of closed positive $(1,1)$-currents is preserved
under small holomorphic deformations by
Corollary~\ref{positive current openness}.

\begin{remark}
For $G_\R$ simple, $X$ is rationally connected by
Corollary~\ref{Mum conj X}, but not rationally connected with compact
deformation space in the sense of \cite[Definition~2.8]{Fujiki}:
Fujiki's theorem in the appendix to that paper would make $X$
Moishezon, contradicting $X\notin\mathcal C$.
See also \cite{Ver} for twistor examples.
\end{remark}

\section{Nonhomogeneous line bundles on \texorpdfstring{$SU(2,1)/H$}{SU(2,1)/H}}
\label{app-nonhomogeneous}

A holomorphic line bundle on a flag domain need not be homogeneous.
We construct topologically trivial line bundles on
$D=SU(2,1)/H$ whose holomorphic isomorphism classes are not
preserved by the $SU(2,1)$ action. This disproves the
homogeneity assertion in
\cite[pp.~240--241, ``Some speculations'', (i)(a)]{GGK}. The
construction uses the projection from a flag to its point.
Its image is the exterior of the closed ball in $\mathbb P^2$,
where we define the bundles by transition functions on two
coordinate charts.

We first describe the domain explicitly.
Let $h(z_0,z_1,w)=|z_0|^2+|z_1|^2-|w|^2$ on $\C^3$, and set
$$
B=\{h<0\}\subset\Pp^2,\qquad
Y=\{h>0\}=\Pp^2\setminus\overline B.
$$
Consider the flag domain
$$
D=\{(p,\ell)\in\Fl(1,2;\C^3):
p\in Y,\ p\in\ell,\ \ell\cap B\ne\varnothing\}.
$$
This is the non-classical flag domain described in
\cite[pp.~109--110]{GGK}. To identify its stabilizer, write
$\ell=p\oplus(\ell\cap p^{\perp_h})$. The first summand is positive
and the second is negative, while $\ell^{\perp_h}$ is positive.
Thus a flag in $D$ determines an orthogonal decomposition with
signs $(+,-,+)$. The group $SU(2,1)$ acts transitively on such
decompositions, with stabilizer $H=S(U(1)^3)$.

\begin{lemma}\label{app-point-projection}
The projection $r:D\to Y$, $(p,\ell)\mapsto p$, admits a global
holomorphic section. In particular,
$r^*:\Pic(Y)\to\Pic(D)$ is injective.
\end{lemma}
\begin{proof}
Fix $q_0=[0:0:1]\in B$ and put
$$
\sigma(p)=(p,\operatorname{span}(p,q_0)).
$$
The line through $p$ and $q_0$ has dual coordinates
$[-z_1:z_0:0]$, which never vanish simultaneously on $Y$.
Thus $\sigma$ is holomorphic, and its image lies in $D$
since this line contains $q_0\in B$.
The identity $r\circ\sigma=\mathrm{id}_Y$ gives the assertion.
\end{proof}

We next construct a family of pairwise nonisomorphic line bundles on $Y$.
The sets $U_i=Y\cap\{z_i\ne0\}$, $i=0,1$, cover $Y$. In coordinates
$$
(z,\xi)=(z_1/z_0,w/z_0),\qquad
(u,\eta)=(z_0/z_1,w/z_1),
$$
they are
$$
U_0=\{(z,\xi):|\xi|^2<1+|z|^2\},\qquad
U_1=\{(u,\eta):|\eta|^2<1+|u|^2\}.
$$
On the overlap, $u=z^{-1}$ and $\eta=\xi/z$. Contracting
$\xi$ and $\eta$ to zero retracts each chart onto $\C$ and
the overlap onto $\C^*$. Hence the charts are simply connected
and the overlap is connected.

For $t\in\C$, define $M_t$ on $Y$ by the transition function
\begin{equation}\label{transition}
g_{01}^{(t)}(z,\xi)=\exp\left(t\frac{\xi^2}{z}\right).
\end{equation}
Then $M_t\otimes M_s\simeq M_{t+s}$. The homotopy
$[z_0:z_1:w]\mapsto[z_0:z_1:aw]$, $0\leq a\leq1$, retracts $Y$
onto $\{w=0\}\simeq\Pp^1$. The transition function is $1$ on
$\{w=0\}$, so $c_1(M_t)=0$.

\begin{lemma}\label{distinct}
For $t,s\in\C$, $M_t\simeq M_s$ if and only if $t=s$.
\end{lemma}
\begin{proof}
It suffices to prove that triviality of $M_t$ implies $t=0$.
A trivialization gives nowhere vanishing holomorphic functions
$h_i$ on $U_i$ such that
$$
\exp(t\xi^2/z)=h_0(z,\xi)/h_1(z^{-1},\xi/z).
$$
Since the charts are simply connected, each $h_i$ has a holomorphic
logarithm $H_i$. On the connected overlap, the difference
$H_0-H_1-t\xi^2/z$ takes values in $2\pi i\mathbb Z$ and is
therefore constant. Adjusting one logarithm gives
$$
t\frac{\xi^2}{z}=H_0(z,\xi)-H_1(z^{-1},\xi/z).
$$
Let $a_0(z)$ and $a_1(u)$ be the coefficients of $\xi^2$ and
$\eta^2$ in the Taylor expansions of $H_0$ and $H_1$ along
$\xi=0$ and $\eta=0$, respectively. They are entire functions.
Comparing these coefficients gives
\begin{equation}\label{residue}
\frac{t}{z}=a_0(z)-z^{-2}a_1(z^{-1}),\qquad z\in\C^*.
\end{equation}
Integrating \eqref{residue} around $|z|=1$ gives $2\pi\sqrt{-1}\,t=0$:
the first term on the right has only nonnegative Laurent powers,
and the second only powers at most $-2$. Hence $t=0$, as required.
\end{proof}

Pulling these bundles back to $D$ gives the required examples.
\begin{theorem}\label{main}
For every $t\ne0$, the line bundle $L_t=r^*M_t$ on $D=SU(2,1)/H$ is
topologically trivial, and $\gamma^*L_t$ is holomorphically trivial
for every holomorphic map $\gamma:\Pp^1\to D$. However, its holomorphic
isomorphism class is not preserved by the $SU(2,1)$ action.
Moreover, $L_t$ admits no smooth Hermitian metric with
$SU(2,1)$-invariant Chern curvature.
\end{theorem}
\begin{proof}
By Lemmas~\ref{app-point-projection} and~\ref{distinct},
$L_t\simeq L_s$ if and only if $t=s$. Consider
$$
g_\theta=\operatorname{diag}(e^{i\theta},e^{i\theta},e^{-2i\theta})
\in SU(2,1).
$$
It fixes $z,u$ and multiplies $\xi,\eta$ by $e^{-3i\theta}$.
Pulling back \eqref{transition} and using the equivariance of $r$ gives
$$
g_\theta^*L_t\simeq L_{e^{-6i\theta}t}.
$$
For $\theta=\pi/6$, this is $L_{-t}\not\simeq L_t$ when $t\ne0$.
Thus $L_t$ is not homogeneous. The same calculation applies
to every covering group acting through $SU(2,1)$.

Since $c_1(L_t)=r^*c_1(M_t)=0$, the bundle is topologically trivial.
For every holomorphic map $\gamma:\Pp^1\to D$, the line bundle
$\gamma^*L_t$ has degree zero and hence is holomorphically trivial.

Finally, the fibration $D\to SU(2,1)/K$ has contractible base and fiber
$K/H\simeq\Pp^1$, so $D$ is simply connected. If $L_t$ admitted a
metric with invariant Chern curvature, the induced metric on
$g^*L_t\otimes L_t^{-1}$ would be flat for every $g\in SU(2,1)$.
This bundle would therefore be holomorphically trivial. Taking
$g=g_{\pi/6}$ contradicts $L_{-t}\not\simeq L_t$.
\end{proof}

\begin{theorem}\label{app-infinite-cohomology}
The vector space $H^1(D,\OO_D)$ is infinite-dimensional.
\end{theorem}
\begin{proof}
For $k\ge2$, put $f_k=\xi^k/z$ on $U_0\cap U_1$. If a finite sum
$\sum_{k=2}^N c_kf_k$ is an additive coboundary, comparison of
the coefficients of $\xi^k$ gives
$$
\frac{c_k}{z}=a_{0,k}(z)-z^{-k}a_{1,k}(z^{-1})
$$
with entire $a_{i,k}$. Since $k\ge2$, the residue on the right is zero,
so $c_k=0$ for every $k$. The Mayer--Vietoris sequence for
$Y=U_0\cup U_1$ contains
$$
H^0(U_0,\OO_Y)\oplus H^0(U_1,\OO_Y)
\longrightarrow H^0(U_0\cap U_1,\OO_Y)
\longrightarrow H^1(Y,\OO_Y).
$$
Its exactness and the coefficient calculation show that the
images of the $f_k$, $k\geq2$, are linearly independent in
$H^1(Y,\OO_Y)$. Finally, $r\circ\sigma=\mathrm{id}_Y$ implies
$\sigma^*r^*=\mathrm{id}$ on $H^1(Y,\OO_Y)$, so $r^*$ embeds this
infinite-dimensional space into $H^1(D,\OO_D)$.
\end{proof}

The homotopy equivalence of $D$ with $K/H\simeq\Pp^1$ gives
$H^1(D,\mathbb Z)=0$. The exponential sequence therefore identifies
the additive group $H^1(D,\OO_D)$ with
$$
\ker\bigl(c_1:\Pic(D)\longrightarrow H^2(D,\mathbb Z)\bigr).
$$
Every line bundle in this kernel restricts trivially along every
holomorphic map $\Pp^1\to D$, since its pullback has degree zero.
Thus triviality on all rational curves does not imply
holomorphic triviality on $D$.
This does not contradict the criterion of Biswas and dos Santos
\cite[Theorem~1.1]{BdS13}: their theorem assumes that the ambient
variety is smooth, projective and separably rationally connected,
whereas $D$ is noncompact.

\appendix

\section{Proof of Proposition~\ref{le D=0 conjecture}}\label{lemma proof}

We prove the root inequality for $D=G_\R/V$ with $G_\R$ simple
and $D$ non-classical. We allow general $V$ and assume that
$\lambda\ne0$ is dominant on $\Delta_+(\mf v,\mf h)$. The compact
root in the conclusion will lie outside $\Delta(\mf v,\mf h)$.
This proves Theorem~\ref{equ of non-classical D}(ii), and gives
Proposition~\ref{le D=0 conjecture} when $V=H$.

We first observe that there are $\beta_1,\beta_2\in\Delta_+^{\mathrm{nc}}$ with
$\beta_1+\beta_2\in\Delta_+^{\mathrm c}$.
Indeed, if no such roots existed, $\mf p_-$ would be abelian.
For $x\in\mf k$ and $u,v\in\mf p_-$, invariance of the Killing
form would then give
$$
B_{\mf g}([x,u],v)=B_{\mf g}(x,[u,v])=0.
$$
The orthogonal complement of $\mf p_-$ in $\mf p$ with respect
to $B_{\mf g}$ is $\mf p_-$ itself, since opposite root spaces
pair nondegenerately. Thus
$[\mf k,\mf p_-]\subseteq\mf p_-$, and similarly for $\mf p_+$.
The decomposition $\mf p=\mf p_-\oplus\mf p_+$ would define an
invariant Hermitian symmetric complex structure on $G_\R/K$
for which $D\to G_\R/K$ is holomorphic. This contradicts
non-classicality; see also \cite[Lemma~1.1]{CT} and
\cite[pp.~77--78]{GGK}.

\begin{lemma}\label{neg in nc}
Assume that $\lambda\ne0$ and
\begin{equation}\label{le D=0 conjecture 0}
\begin{aligned}
(\lambda,\alpha)&\geq0 &&\text{for all }\alpha\in\Delta_+^{\mathrm c},\\
(\lambda,\beta)&\leq0 &&\text{for all }\beta\in\Delta_+^{\mathrm{nc}}.
\end{aligned}
\end{equation}
Then there are $\beta',\beta''\in\Delta_+^{\mathrm{nc}}$ such that
\begin{equation}\label{le D=0 conjecture 1}
(\lambda,\beta')<0,\qquad (\lambda,\beta'')=0.
\end{equation}
\end{lemma}
\begin{proof}
If the conclusion fails, either
\begin{equation}\label{le D=0 conjecture 2}
(\lambda,\beta)<0\quad\text{for all }\beta\in\Delta_+^{\mathrm{nc}},
\end{equation}
or
\begin{equation}\label{le D=0 conjecture 3}
(\lambda,\beta)=0\quad\text{for all }\beta\in\Delta_+^{\mathrm{nc}}.
\end{equation}
In the first case the roots $\beta_1,\beta_2$ above give
$(\lambda,\beta_1+\beta_2)<0$, contrary to
\eqref{le D=0 conjecture 0}.

For the second case, recall that
\begin{equation}\label{k=p+p}
\mf k=[\mf p,\mf p].
\end{equation}
Indeed, the Cartan bracket relations and the Jacobi identity show
that $[\mf p_0,\mf p_0]\oplus\mf p_0$ is a nonzero ideal of
$\mf g_0$. Since $\mf g_0$ is simple, this ideal equals $\mf g_0$.
Complexifying gives \eqref{k=p+p}.
Every compact root space therefore occurs as a bracket of two
noncompact root spaces. If $\lambda$ annihilates the noncompact
roots, it also annihilates the compact roots. Since the roots span
$\sqrt{-1}\mf h_0^*$, this would give $\lambda=0$.
\end{proof}

We shall also use the following consequence of the Jacobi identity.

\begin{lemma}\label{3 to 2}
Let $\alpha,\beta,\gamma\in\Delta$. Suppose that
$\alpha+\beta+\gamma$ and $\beta+\gamma$ are roots and that
$\alpha+\beta\ne0$ and $\alpha+\gamma\ne0$.
Then at least one of $\alpha+\beta$ and $\alpha+\gamma$ is a root.
\end{lemma}
\begin{proof}
Choose nonzero root vectors $e_\alpha,e_\beta,e_\gamma$.
The assumptions on $\beta+\gamma$ and $\alpha+\beta+\gamma$ imply
$[e_\alpha,[e_\beta,e_\gamma]]\ne0$.
If neither $\alpha+\beta$ nor $\alpha+\gamma$ were a root, the
other assumptions would give
$[e_\alpha,e_\beta]=[e_\alpha,e_\gamma]=0$.
The Jacobi identity would then give
$[e_\alpha,[e_\beta,e_\gamma]]=0$, a contradiction.
\end{proof}

\begin{proof}[Proof of Proposition~\ref{le D=0 conjecture}]
Suppose that the asserted inequalities fail. Dominance on
$\Delta_+(\mf v,\mf h)$ then gives \eqref{le D=0 conjecture 0}.
We shall construct a nonzero ideal $\mf I$ that is stable under
conjugation with respect to $\mf g_0$ and for which
$\mf I\cap\mf p_-$ is abelian. Simplicity of $\mf g_0$ will
then imply $\mf I=\mf g$, making $\mf p_-$ abelian and
contradicting the non-classicality of $D$.

Let
$$
\Delta_+^{\mathrm{nc},1}
=\{\beta\in\Delta_+^{\mathrm{nc}}:(\lambda,\beta)<0\}.
$$
Let $\Delta_+^{\mathrm{nc},2}$ be the set of roots
$\beta\in\Delta_+^{\mathrm{nc}}$ with $(\lambda,\beta)=0$ that
can be written as $\alpha+\beta'$, where
$\alpha\in\Delta_+^{\mathrm c}$ and
$\beta'\in\Delta_+^{\mathrm{nc},1}$. Put
$$
\Delta_+^{\mathrm{nc},3}
=\Delta_+^{\mathrm{nc}}\setminus
 (\Delta_+^{\mathrm{nc},1}\cup\Delta_+^{\mathrm{nc},2}).
$$
By Lemma~\ref{neg in nc}, $\Delta_+^{\mathrm{nc},1}$ is a nonempty
proper subset of $\Delta_+^{\mathrm{nc}}$. If
$\beta=\alpha+\beta'\in\Delta_+^{\mathrm{nc},2}$, then
$(\lambda,\alpha)=-(\lambda,\beta')>0$.
Put
$$
\Sigma=\Delta_+^{\mathrm{nc},1}\cup\Delta_+^{\mathrm{nc},2},
\qquad
\mf I_- =\bigoplus_{\beta\in\Sigma}\mf g_\beta,
\qquad \mf I_+=\bar{\mf I_-},
\qquad \mf I_0=[\mf I_+,\mf I_-].
$$
We shall show that $\mf I=\mf I_+\oplus\mf I_0\oplus\mf I_-$
is an ideal by examining sums and differences of roots.
For subsets $A,B\subset\Delta$, write
$A\mathbin{\pm_\Delta}B=(A\pm B)\cap\Delta$.

\emph{Sums of noncompact roots.}
If $\beta\in\Delta_+^{\mathrm{nc},1}$ and
$\beta'\in\Delta_+^{\mathrm{nc}}$, their sum cannot be a root:
it would be compact and have negative pairing with $\lambda$.
Thus
\begin{equation}\label{nc1 + nc=0}
\Delta_+^{\mathrm{nc},1}+_\Delta\Delta_+^{\mathrm{nc}}=\varnothing.
\end{equation}
Now let $\gamma$ be a positive noncompact root with
$(\lambda,\gamma)=0$, and write
$\delta=\alpha+\beta\in\Delta_+^{\mathrm{nc},2}$, where
$\alpha\in\Delta_+^{\mathrm c}$ and
$\beta\in\Delta_+^{\mathrm{nc},1}$.
If $\gamma+\delta$ were a root, Lemma~\ref{3 to 2} would make
$\gamma+\alpha$ or $\gamma+\beta$ a root. Both partial sums
are nonzero. The second is excluded by \eqref{nc1 + nc=0};
the first would be positive noncompact with pairing
$(\lambda,\alpha)>0$, which is also impossible.
Together with \eqref{nc1 + nc=0}, this proves
$$
\Sigma+_\Delta\Delta_+^{\mathrm{nc}}=\varnothing.
$$
In particular,
\begin{equation}\label{I remaining}
\Delta_+^{\mathrm{nc},2}+_\Delta\Delta_+^{\mathrm{nc},2}
=\varnothing.
\end{equation}

\emph{Differences between roots in $\Delta_+^{\mathrm{nc},3}$ and $\Sigma$.}
We claim that
$$
\Delta_+^{\mathrm{nc},3}-_\Delta\Sigma=\varnothing.
$$
Let $\gamma\in\Delta_+^{\mathrm{nc},3}$ and first take
$\beta\in\Delta_+^{\mathrm{nc},1}$.
If $\alpha=\gamma-\beta$ is a root, it is compact and
$(\lambda,\alpha)=-(\lambda,\beta)>0$. Hence $\alpha$ is
positive, and $\gamma=\alpha+\beta$ belongs to
$\Delta_+^{\mathrm{nc},2}$, a contradiction.
Next write $\delta=\alpha+\beta\in\Delta_+^{\mathrm{nc},2}$
as above. If $\gamma-\delta$ is a root, apply
Lemma~\ref{3 to 2} to $\gamma,-\alpha,-\beta$.
Here $\gamma-\alpha\ne0$, since $\gamma$ is noncompact and
$\alpha$ is compact. Also $\gamma-\beta\ne0$, since
$(\lambda,\gamma)=0$ and $(\lambda,\beta)<0$.
The preceding case excludes $\gamma-\beta$, so
$\gamma-\alpha$ is a noncompact root. Since
$(\lambda,\gamma-\alpha)=-(\lambda,\alpha)<0$, it is positive
and belongs to $\Delta_+^{\mathrm{nc},1}$. Again
$\gamma=\alpha+(\gamma-\alpha)$ belongs to the second set,
a contradiction.

\emph{Adding and subtracting compact roots.}
We next prove
\begin{align*}
\Delta_+^{\mathrm c}+_\Delta\Sigma&\subseteq\Sigma,\tag{III.1}\\
\Delta_+^{\mathrm c}-_\Delta\Sigma&\subseteq-\Sigma.\tag{III.2}
\end{align*}
For (III.1), take $\alpha'\in\Delta_+^{\mathrm c}$.
If $\beta\in\Delta_+^{\mathrm{nc},1}$ and $\alpha'+\beta$
is a root, its pairing is nonpositive, so it belongs to the
first set or, by definition, to the second.
If instead $\delta=\alpha+\beta\in\Delta_+^{\mathrm{nc},2}$
and $\alpha'+\delta$ is a root, its pairing forces
$(\lambda,\alpha')=0$.
Lemma~\ref{3 to 2} gives a root $\alpha'+\alpha$ or
$\alpha'+\beta$. In the first case,
$\alpha'+\delta=(\alpha'+\alpha)+\beta$ is in the second set.
In the second case, $\alpha'+\beta$ has negative pairing,
so $\alpha'+\delta=\alpha+(\alpha'+\beta)$ is again in the
second set. This proves (III.1).

If $\beta\in\Delta_+^{\mathrm{nc},1}$ and
$\alpha'-\beta$ is a root, it is noncompact and has positive
pairing. It is therefore negative, with its opposite in the
first set. Thus
\begin{equation}\label{III.2.1}
\Delta_+^{\mathrm c}-_\Delta\Delta_+^{\mathrm{nc},1}
\subseteq-\Delta_+^{\mathrm{nc},1}.
\end{equation}
For $\delta=\alpha+\beta\in\Delta_+^{\mathrm{nc},2}$,
assume that $\alpha'-\delta$ is a root.
If $(\lambda,\alpha')>0$, then $\alpha'-\delta$ has positive
pairing with $\lambda$ and therefore belongs to
$-\Delta_+^{\mathrm{nc},1}$.
If $(\lambda,\alpha')=0$, then $\alpha'-\alpha\ne0$ and
$\alpha'-\beta\ne0$, so Lemma~\ref{3 to 2} gives one of these
as a root. In the first case,
$\alpha'-\alpha=-\alpha''$ for a positive compact root
$\alpha''$, and
$\alpha'-\delta=-(\alpha''+\beta)$ belongs to
$-\Delta_+^{\mathrm{nc},2}$.
In the second case \eqref{III.2.1} gives
$\alpha'-\beta=-\beta'$ with
$\beta'\in\Delta_+^{\mathrm{nc},1}$, and
$\alpha'-\delta=-(\alpha+\beta')$ has the same property.
We have proved
\begin{equation}\label{III.2.2}
\Delta_+^{\mathrm c}-_\Delta\Delta_+^{\mathrm{nc},2}
\subseteq-\Sigma,
\end{equation}
and hence (III.2).

The sum and difference relations for noncompact roots give
$$
[\mf p_-,\mf I_-]=0,
\qquad [\mf p_-,\mf I_+]=[\mf I_-,\mf I_+]=\mf I_0.
$$
Relations (III.1)--(III.2) and their conjugates, together with
$[\mf h,\mf I_\pm]\subseteq\mf I_\pm$, give
$[\mf k,\mf I_\pm]\subseteq\mf I_\pm$.
Since $\mf g=\mf p_+\oplus\mf k\oplus\mf p_-$, it follows that
$$
[\mf g,\mf I_\pm]\subseteq\mf I_0\oplus\mf I_\pm.
$$
Also $\mf I_0\subseteq\mf k$, and the Jacobi identity gives
\begin{align*}
[\mf g,\mf I_0]
&\subseteq[\mf I_+,[\mf g,\mf I_-]]
       +[\mf I_-,[\mf g,\mf I_+]]\\
&\subseteq[\mf I_+,\mf I_0\oplus\mf I_-]
       +[\mf I_-,\mf I_0\oplus\mf I_+]\subseteq\mf I.
\end{align*}
Thus $\mf I$ is an ideal. It is nonzero because
$\Delta_+^{\mathrm{nc},1}\ne\varnothing$, and it is stable under
the conjugation defined by $\mf g_0$.
Consequently $\mf I\cap\mf g_0$ is a nonzero real ideal, so
simplicity of $\mf g_0$ implies $\mf I=\mf g$.
In particular, $\mf I_-=\mf p_-$.
The relation $[\mf p_-,\mf I_-]=0$ now says that $\mf p_-$
is abelian, contrary to the observation at the beginning of
the appendix.
\end{proof}

\end{document}